\documentclass[11pt]{article}
\usepackage{latexsym}  
\usepackage{enumerate}
\usepackage{amsmath,amssymb,amsfonts,graphicx,tikz,tikz-cd}
\usepackage{enumitem}
\usepackage{color}

\begin{document}

\title{Left restriction monoids from left $E$-completions}
\author{Tim Stokes}

\date{}
\maketitle

\newcommand{\bea}{\begin{eqnarray*}}
\newcommand{\eea}{\end{eqnarray*}}

\newcommand{\ben}{\begin{enumerate}}
\newcommand{\een}{\end{enumerate}}

\newcommand{\bi}{\begin{itemize}}
\newcommand{\ei}{\end{itemize}}

\newcommand{\tie}{\bowtie}

\newenvironment{proof}{\noindent \textbf{Proof.}\hspace{.7em}}
                   {\hfill $\Box$
                    \vspace{10pt}}

\newcommand{\mc}{\mathcal}
\newcommand{\dom}{\mbox{dom}}
\newcommand{\ran}{\mbox{ran}}

\newtheorem{thm}{Theorem}[section]
\newtheorem{theorem}[thm]{Theorem}
\newtheorem{lem}[thm]{Lemma}
\newtheorem{pro}[thm]{Proposition}
\newtheorem{dfnpro}[thm]{Definition/Proposition}
\newtheorem{proposition}[thm]{Proposition}
\newtheorem{cor}[thm]{Corollary}
\newtheorem{conj}[thm]{Conjecture}
\newtheorem{corollary}[thm]{Corollary}
\newtheorem{eg}[thm]{Example}
\newtheorem{dfn}[thm]{Definition}

\newcommand{\gl}{{\mc L}}
\newcommand{\gr}{{\mc R}}

\newcommand{\C}{{\mc C}}

\def\Z{{\mathbf Z}}
\def\F{{\mathbf F}}
\def\R{{\mathbf R}}

\newcommand\nc\newcommand
\nc\rnc\renewcommand

\nc{\uv}[1]{\fill (#1,2)circle(.17);}
\nc{\lv}[1]{\fill (#1,0)circle(.17);}
\nc{\mv}[1]{\fill (#1,1)circle(.17);}  
\nc{\uvs}[1]{{\foreach \x in {#1} { \uv{\x}}}}
\nc{\mvs}[1]{{\foreach \x in {#1} { \mv{\x}}}}
\nc{\lvs}[1]{{\foreach \x in {#1} { \lv{\x}}}}
\nc\colrect[5]{\fill[#5!20](#1,2)--(#2,2)--(#4,0)--(#3,0);}
\nc\colrecthigh[3]{\fill[#3!20](#1,1.75)--(#2,1.75)--(#2,2)--(#1,2);}
\nc\colrectlow[3]{\fill[#3!20](#1,.25)--(#2,.25)--(#2,0)--(#1,0);}
\nc{\uvw}[1]{\draw[fill=white] (#1,2)circle(.18);}
\nc{\lvw}[1]{\draw (#1,0)circle(.18);}
\nc{\uvws}[1]{\foreach \x in {#1}{ \uvw{\x}}}
\nc{\lvws}[1]{\foreach \x in {#1}{ \lvw{\x}}}
\nc{\uverts}[1]{\foreach \x in {#1}{ \uvert{\x}}}
\nc{\lverts}[1]{\foreach \x in {#1}{ \lvert{\x}}}
\nc{\uarcs}[1]{
{\foreach \x/\y in {#1}
{ \uarc{\x}{\y} }
}
}

\nc{\uvert}[1]{\fill (#1,2)circle(.2);}
\rnc{\lvert}[1]{\fill (#1,0)circle(.2);}
\nc{\uarcx}[3]{\draw(#1,2)arc(180:270:#3) (#1+#3,2-#3)--(#2-#3,2-#3) (#2-#3,2-#3) arc(270:360:#3);}
\nc{\uarc}[2]{\uarcx{#1}{#2}{.4}}
\nc{\darcx}[3]{\draw(#1,0)arc(180:90:#3) (#1+#3,#3)--(#2-#3,#3) (#2-#3,#3) arc(90:0:#3);}
\nc{\darc}[2]{\darcx{#1}{#2}{.4}}

\nc{\marcx}[3]{\draw(#1,1)arc(180:90:#3) (#1+#3,1+#3)--(#2-#3,1+#3) (#2-#3,1+#3) arc(90:0:#3);} 
\nc{\marc}[2]{\marcx{#1}{#2}{.4}}

\nc{\uarcxt}[3]{\draw(#1,4)arc(180:270:#3) (#1+#3,4-#3)--(#2-#3,4-#3) (#2-#3,4-#3) arc(270:360:#3);} 
\nc{\uarct}[2]{\uarcx{#1}{#2}{.4}}
\nc{\marcxt}[3]{\draw(#1,2)arc(180:90:#3) (#1+#3,2+#3)--(#2-#3,2+#3) (#2-#3,2+#3) arc(90:0:#3);}  
\nc{\marct}[2]{\marcx{#1}{#2}{.4}}

\nc{\stline}[2]{\draw(#1,2)--(#2,0);}
\nc{\stlinea}[2]{\draw(#1,2)--(#2,1);}  
\nc{\stlineb}[2]{\draw(#1,1)--(#2,0);} 

\nc{\custpartn}[3]{{\lower1.4 ex\hbox{
\begin{tikzpicture}[scale=.3]
\foreach \x in {#1}
{ \uvert{\x}  }
\foreach \x in {#2}
{ \lvert{\x}  }
#3 \end{tikzpicture}
}}}

\begin{abstract}  
Given a monoid $S$ with $E$ any non-empty subset of its idempotents, we present a novel one-sided version of idempotent completion we call left $E$-completion.  In general, the construction yields a one-sided variant of a small category called a constellation by Gould and Hollings.  Under certain conditions, this constellation is inductive, meaning that its partial multiplication may be extended to give a left restriction semigroup, a type of unary semigroup whose unary operation models domain.  We study the properties of those pairs $S,E$ for which this happens, and characterise those left restriction semigroups that arise as such left $E$-completions of their submonoid of elements having domain $1$.  As first applications, we decompose the left restriction semigroup of partial functions on the set $X$ and the  right restriction semigroup of left total partitions on $X$ as left and right $E$-completions respectively of the transformation semigroup $T_X$ on $X$, and decompose the left restriction semigroup of binary relations on $X$ under demonic composition as a left $E$-completion of the left-total binary relations.  In many cases, including these three examples, the construction embeds in a semigroup Zappa-Sz\'{e}p product.  
\end{abstract}

\medskip


\section{Introduction}  \label{intro}

The full transformation semigroup $T_X$ on a non-empty set $X$ is perhaps the most important example of a semigroup.  However, the larger monoid $PT_X$ of partial transformations on $X$ is amongst the most important.  The monoid $PT_X$ is frequently viewed as a unary semigroup, equipped with unary operation $D$ given by $D(f)=\{(x,x)\mid x\in \dom(f)\}$, the identity map on $\dom(f)$.  When equipped with $D$ in this way, $PT_X$ is a {\em left restriction semigroup}.  The class of left restriction semigroups is an equational class of unary semigroups, and $PT_X$ is canonical in the sense that all embed in one such: see \cite{trok}, as well as \cite{Csemi} and \cite{manes}.  Left restriction semigroups have received increasing attention in recent years.

There are other examples of left restriction semigroups.  The monoid of binary relations $Rel_X$ under relational composition may be equipped with a domain operation $D$ defined as for $PT_X$, but the resulting unary semigroup is not a left restriction semigroup (although it satisfies most of their defining laws,  and indeed when equipped with the operation of range $R$, defined in the obvious dual way to $D$, is an Ehresmann semigroup; see \cite{Lawson}).  However, if one replaces the usual composition operation by so-called demonic composition $\circledast$, $(Rel_X,\circledast,D)$ is a left restriction semigroup.  Demonic composition is of interest in computer science because of its connection with total correctness of (possibly non-deterministic) programs; see \cite{DAD} for example.  A further example, on this occasion a right restriction semigroup, arises from the partition monoid $P_X$ on a set $X$, namely the submonoid $P^{lt}_X$ of $P_X$ consisting of left total partitions (defined in the natural way).

In what follows, we show that each of the families of examples just mentioned arises via a straightforward construction we call left (or sometimes right) $E$-completion.  It follows that each can be built from the submonoid of elements of domain (or perhaps range) $1$, together with some suitable idempotents $E$ of this submonoid.  We characterise those left restriction semigroups that arise from such a construction.  In particular, $PT_X$ can be constructed from $T_X$ using this construction, as can $Rel_X$ under demonic composition and domain.  But, more surprisingly, $P^{lt}_X$ can be obtained from $T_X$ using the right-handed version of the construction.  

We now describe the construction in more detail.  For any non-empty $E\subseteq E(S)$, define 
$$Cat_E(S)=\{(e,s,f)\in E\times S\times E\mid esf=s\}.$$  
Define the partial binary operation $\circ$ on $Cat_E(S)$ by setting 
$$(e,s,f)\circ (f,t,g):=(e,st,g),$$
with no other products defined.  
Also define $D((e,s,f))=(e,e,e)$ and $R((e,s,f))=(f,f,f)$.  It is routine to check that $\circ$ is well-defined, and that $(Cat_E(S),\circ,D,R)$ is a small category with domain and range operations $D$ and $R$ respectively.
We call this category the {\em $E$-completion of $S$}.  If $E=E(S)$, this is nothing but the Karoubi envelope or idempotent completion of the semigroup; see Part C of \cite{tilson}, where this is defined for an arbitrary semigroup, though not under either of the above names.  Prior to \cite{tilson}, the concept was in use by category theorists as a way of canonically embedding a category into a larger one in which every idempotent splits.

We are interested in a one-sided version of $E$-completion.  

\begin{dfn}  \label{CES}
Let $S$ be a semigroup, with $E\subseteq E(S)$.  Define 
$$C_E(S)=\{(e,s)\in E\times S\mid es=s\}.$$
Define the partial binary operation $\circ$ on $C_E(S)$ by setting 
$(e,s)\circ (f,t)=(e,st)\mbox{ whenever }sf=s$, and define $D((e,s))=(e,e)$. 
The structure $(C_E(S),\circ,D)$ is the {\em left $E$-completion} of $S$. 
\end{dfn}

The left $E$-completion $C_E(S)$ is not a category, but it is a {\em constellation}, in the sense defined by Gould and Hollings in \cite{inductive} and then considered further in \cite{actions} and \cite{constgen}.  In \cite{inductive}, constellations were defined in order to describe partial algebraic counterparts of left restriction semigroups.  This followed earlier work in which inverse semigroups were shown to correspond to inductive groupoids by Ehresmann and Schein (see \cite{ehres}, \cite{schein}), as well as work due to Lawson showing that so-called Ehresmann semigroups correspond to so-called Ehresmann categories (see \cite{Lawson}), a special case of which is the fact that two-sided restriction semigroups correspond to inductive categories.  Ehresmann and restriction semigroups have both domain and range operations, while left restriction semigroups have only domain, so the first task for Gould and Hollings was to replace categories by some one-sided analog, namely constellations.  Constellations have an (asymmetric) partial binary composition operation and a domain operation but no range. 

In what follows, we make use of the notion of a demigroup introduced in \cite{demonize}, in order to consider variants of left $E$-completions as just defined.  

\begin{dfn}  \label{demidef}
A {\em demigroup} $S$ is a unary semigroup with unary operation $d$ mapping into $E(S)$ and satisfying the two laws $d(x)x=x$ and $d(xy)=d(xd(y))$.
\end{dfn} 

Demigroups were defined in \cite{demonize} based on the most general way in which a constellation can arise from a unary semigroup in a sense we describe in detail later.

\begin{dfn}
Let $S$ be a demigroup, with $E\subseteq E(S)$ such that $d(s)\in E$ for all $s\in S$, and $ed(e)=e$ for all $e\in E$. Then we say $S$ is an {\em $E$-demigroup}.  Let
$$C^d_E(S)=\{(e,s)\in E\times S\mid es=s, d(e)=d(s)\}.$$
Define the partial binary operation $\circ$ on $C^d_E(S)$ as well as the unary operation $D$ as for $C_E(S)$, that is, 
$(e,s)\circ (f,t)=(e,st)\mbox{ whenever }sf=s$, and $D((e,s))=(e,e)$. 
The structure $(C^d_E(S),\circ,D)$ is the {\em left $(d,E)$-completion} of $S$. 
\end{dfn}

We shall show that $C^d_E(S)$ is a subconstellation of $C_E(S)$ in general.  But there are two special cases of the greatest importance in what follows.

First, note that if $S$ is a monoid, we may select any set of idempotents $E\subseteq E(S)$ that contains $1$, define $d(s)=1$ for all $s\in S$, and the result is an $E$-demigroup; conversely, every monoid $S$ which is an $E$-demigroup in which $d(s)=1$ for all $s\in S$ clearly arises in this way.  We give such cases a special name.

\begin{dfn}
Let $S$ be a monoid which is an $E$-demigroup, with $1\in E\subseteq E(S)$.  
If $d(s)=1$ for all $s\in S$, then we say $S$ is a {\em left $E$-monoid}. 
\end{dfn}

Hence, a left $E$-monoid is equivalent to a monoid with a distinguished set of idempotents, to which $1$ is added to give $E$.

Secondly, suppose $S$ arises from adjoining a zero to a monoid, and has distinguished $E\subseteq E(S)$ such that $0,1\in E$.   Defining $d(0)=0$ and $d(s)=1$ otherwise, gives an $E$-demigroup; conversely, if $S$ is an $E$-demigroup having zero $0$ and identity $1$ with $d(0)=0$ and $d(s)=1$ otherwise, then $S$ must arise from adjoining a zero to a monoid; we show this in detail below.

\begin{dfn}  \label{lem0}
If $S$ is a monoid with zero that is an $E$-demigroup such that $d(0)=0$ with $d(s)=1$ if $s\neq 0$, we say $S$ is a {\em left $E$-monoid with zero}, and we also denote $C^d_E(S)$ by $C^0_E(S)$.
\end{dfn}

So a left $E$-monoid with zero is equivalent to a monoid with zero having a distinguished set of idempotents, to which $0$ and $1$ are added to give $E$.

If $S$ is a left $E$-monoid, then clearly $C^d_E(S)=C_E(S)$, while if $S$ is a left $E$-monoid with zero, note that $C^0_E(S)=\{(e,s)\in C_E(S)\mid s=0\Rightarrow e=0\}$.

Mostly we are interested in left $E$-monoids and left $E$-monoids with zero.  The motivation behind our general definition of $C^d_E(S)$ in terms of $E$-demigroups was to avoid the necessity of several slightly different definitions and proofs involving $C_E(S)$ and $C^0_E(S)$ for these two cases.  As an added bonus, the case in which $S$ is a demigroup and we set $E=d(S)=\{d(s)\mid s\in S\}$, giving $C^d_{d(S)}(S)$, allows the main construction of \cite{demonize} to be viewed as a special case of the current approach.  Other special cases may prove of interest in the future.

As outlined above, the reason for introducing constellations in \cite{inductive} was to retain from a left restriction semigroup only those products that were needed to re-construct it from the partial algebraic structure, as well as the natural partial order determined by the left restriction semigroup (although this can also be defined algebraically within the constellation).  But one can reverse this.  In the current case, a natural question is: when is the constellation $C^d_E(S)$ inductive, hence nothing but a ``left restriction semigroups in disguise"?  And when does a given left restriction semigroup arise in this way?  In what follows, we obtain a complete answer to the first question by characterising exactly when $C^d_E(S)$ is inductive in terms of conditions on the $E$-demigroup $S$.  We then characterise those left restriction semigroups that arise as $C_E(S)$ or $C^0_E(S)$, and show that important examples of left restriction semigroups arise in this way. 

In Section \ref{prelims} to follow, we work through some preliminaries, with discussion of constellations and left restriction semigroups and the connections between them, as well as of demigroups.  Section \ref{extend} considers $E$-demigroups and their left $(d,E)$-completions, and we establish exactly when an $E$-demigroup $S$ gives rise to an inductive constellation  $C^d_E(S)$, with particular reference to the special cases $C_E(S)$ and $C^0_E(S)$.  Then in Section \ref{lrchar}, we give an internal characterisation of those left restriction semigroups that arise as either $C_E(S)$ or $C^0_E(S)$, and the pairs $S,E$ satisfying the necessary and sufficient conditions are considered for their own sake in Section \ref{leftEmod}.  In Section \ref{Zappa}, we make a connection with semigroup Zappa-Sz\'{e}p products.

In Section \ref{egs}, we give three applications.  It is shown that the left restriction semigroup $(PT_X,\cdot,D)$ (where $\cdot$ is composition) is the zero-reduced left $E$-completion of $(T_X^0,\cdot)$ (which is $T_X$ with adjoined zero element) for any $E\subseteq E(S)$ that is maximal right pre-reduced (defined below).  The left restriction semigroup of binary relations $Rel_X$ under demonic composition, $(Rel_X,\circledast,D)$, is then shown to be the zero-reduced left $E$-completion of the submonoid of left total binary relations for the same choice of $E\subseteq T_X^0\subseteq Rel_X^0$ as in the previous example.  Finally, and perhaps most unexpectedly, the right restriction semigroup $(P^{lt}_X,\cdot,R)$ consisting of left total partitions on $X$ is shown to be the {\em right} $E$-completion of $T_X$ using any maximal {\em left} pre-reduced set of idempotents $E$.  It is shown that with careful choice of $E$, each of these constructions can be viewed as embedded in a semigroup Zappa-Sz\'{e}p product.

In the final section we briefly discuss some directions for future work.

\section{Preliminaries}  \label{prelims}

The current work combines many topics, so we begin with a review of them.   

First, some notation.  In what follows, within semigroups consisting of functions or relations, all function and relation compositions will be written left to right, so that ``$fg$" denotes ``first $f$, then $g$", and correspondingly, we write ``$xf$" when the function $f$ is applied to $x\in \dom(f)$, rather than ``$f(x)$".  The only exception is when unary operations are used: if $D$ is unary, we write ``$D(s)$" rather than ``$sD$".  Recall that if $S$ is a semigroup then $E(S)$ denotes its set of idempotents.

\subsection{Some properties of idempotents}  \label{preredsec}

Throughout this subsection, $S$ is a semigroup.  Two familiar natural quasiorders on $E(S)$ are $\omega^l$ and $\omega^r$,  given by $(e,f)\in \omega^l$ if and only if $e=ef$, and $(e,f)\in \omega^r$ if and only if $e=fe$.  The respective induced equivalence relations on $E(S)$ are evidently Green's relations ${\mc L}$ and ${\mc R}$ respectively, restricted to $E(S)$, and we denote them this way throughout.   The {\em natural order} $\omega$ on $E(S)$ is the intersection of $\omega^l$ and $\omega^r$, so that $(e,f)\in \omega$ if and only if $e=ef=fe$, and is a partial order on $E(S)$, with respect to which $1$ is the top element if $S$ is a monoid with identity $1$, and if $S$ has a zero $0$ then it is the least element.  (The notation $\leq_r$ and $\leq_l$ rather than $\omega^l$ and $\omega^r$ is used in \cite{gendom} and \cite{demonize}, but here we use the more standard notation of semigroup theory.)

Let $E$ be a nonempty subset of $E(S)$. All of the following concepts are defined in \cite{gendom}, although many of them were first defined earlier than that.  We say $E$ is {\em right pre-reduced} if $\omega^l$ is a partial order (that is, $e=ef$ and $f=fe$ imply $e=f$), and {\em left pre-reduced} if  $\omega^r$ is a partial order.  Clearly, $E\subseteq E(S)$ is right pre-reduced if and only if no two elements of $E$ are in the same ${\mc L}$-class.  

We say that $E$ is {\em right reduced} if for all $e,f\in E$, $e=ef$ implies $e=fe$.  There is an obvious dual notion of {\em left reduced}.  We say $E$ is {\em reduced} if it is both left reduced and right reduced; this concept was defined in \cite{Lawson}.  If the elements of $E$ commute with one-another then $E$ is evidently reduced.  Obviously, if $E$ is right reduced, then it is right pre-reduced, although the converse fails in general.  So $E$ is right reduced if ${\omega^l}\subseteq {\omega^r}$ (which happens if and only if $\omega^l$ is the natural order on $E$), left reduced if the opposite inclusion holds, and reduced if ${\omega^l}={\omega^r}$.  

\begin{dfn}
Let $S$ be a semigroup.  We say that the two right pre-reduced sets $E$ and $E'\subseteq E(S)$ are {\em right-equivalent}, and write $E\sim_l E'$, if there is a bijection $e\mapsto e'$ from $E$ to $E'$ such that $e\, {\mc L}\, e'$.  {\em Left-equivalence} is defined dually.
\end{dfn}

If two right pre-reduced sets $E,E'\subseteq E(S)$ are right-equivalent, then it is easily seen that $(E,\omega^l)\cong (E',\omega^l)$ as posets.  

If we pick precisely one element of $E(S)$ from each ${\mc L}$-class to give $F\subseteq E(S)$, then $F$ is right pre-reduced; moreover every right pre-reduced $E\subseteq E(S)$ may be extended to such $F$.  We therefore call such $F$ a {\em maximal right pre-reduced subset of $E(S)$}.  Obviously, any two maximal right pre-reduced subsets are right-equivalent.  All these comments and definitions have dual left-sided versions.

The following elementary facts (which are easy consequences of Green's relations) will be useful in what follows.

\begin{lem}  \label{bighelp}
Let $S$ be a semigroup.  Pick $s,t\in S$ and $e,e',f,f'\in E(S)$ with $e\, {\mc L}\, e'$ and $f\, {\mc L}\, f'$.
\ben
\item If $se=s$ then $se'=s$.
\item If $es=et$ then $e's=e't$.
\item If $esf=es$ then $e'sf'=e's$.
\een
\end{lem}

For $S$ a semigroup and $a\in S$, denote by $Sa=\{sa\mid s\in S\}$, the left ideal generated by $a\in S$.  If $a\in E(S)$ or if $S$ is a monoid, then $a\in Sa$.  

The following is trivial and we commented upon it earlier.

\begin{lem}  \label{lisame}
Let $S$ be a semigroup.  Then for $e,f\in E(S)$, $Se=Sf$ if and only if $e\, {\mc L}\, f$.
\end{lem}

\begin{dfn} 
In a semigroup $S$, define $Eq(s,t)=\{u\in S\mid us=ut\}$ for all $s,t\in S$.  If non-empty, $Eq(s,t)$ is a left ideal of $S$, the {\em (left) equalizing set} of $s,t$.  If $Eq(s,t)=Se$ for some $e\in E(S)$, we call $e$ a {\em left equalizer} of $(s,t)$, and we say $e$ {\em equalizes} $(s,t)$.   Dually, define $REq(s,t)=\{u\in S\mid su=tu\}$, a right ideal of $S$ when non-empty, and define {\em right equalizers} dually to left equalizers.
\end{dfn}  

The right-handed version of this concept was considered in \cite{ghs} and called $r^S(s,t)$; semigroups $S$ for which $r^S(s,t)$ is finitely generated characterise a certain type of $S$-act axiomatisability as in \cite{gmps}, and such semigroups are studied for their own sake in \cite{ghs}.

Obviously, this notion of left equalizer is different to the category notion.

\begin{pro}  \label{eq}
Let $S$ be a semigroup with $s,t\in S$, and suppose $e$ equalizes $(s,t)$.  Then $f$ equalizes $(s,t)$ if and only if $e\, {\mc L}\, f$.  Moreover, for all $s\in S$ and $e,f\in E(S)$ for which $e\, {\mc L}\, f$, $Eq(s,se)=Eq(s,sf)$.
\end{pro}
\begin{proof}
If $e$ equalizes $(s,t)$ then $Eq(s,t)=Se$.  The following are then equivalent by Lemma \ref{lisame}: $f$ equalizes $(s,t)$; $Se=Sf$; $e\, {\mc L}\, f$. 

The final statement follows because for all $s,t\in S$ and $e,f\in E(S)$ with $e\, {\mc L}\, f$, $tse=ts$ if and only if $tsf=ts$ by (1) of Lemma \ref{bighelp}. 
\end{proof}

\subsection{Left restriction semigroups and their generalisations}  \label{leftrestetc}

The semigroups of most interest to us here are unary semigroups equipped with an operation $D$, called {\em domain}.  The class of left restriction semigroups satisfies the following laws (see \cite{inductive}): 
\begin{enumerate}[label=\textup{(R\arabic*)}]

\item $D(x)x=x$;
\item $D(x)D(y)=D(y)D(x)$;
\item $D(D(x)y)=D(x)D(y)$;
\item $xD(y)=D(xy)x$.  \label{r4}
\een
Other useful laws follow easily such as 
\ben
\item[\textup{(R5)}]
 $D(xy)=D(xD(y))$ (the left congruence condition). 
\een

Other equivalent axioms have been used by various authors.  There is an obvious dual notion of right restriction semigroup; the unary operation is typically denoted $R$ (``range") in such cases.
  
It can easily be shown that in a left restriction semigroup $S$, $D(S)=\{D(s)\mid s\in S\}$ is a semilattice under multiplication, and $D(s)$ is the smallest $e\in D(S)$ (with respect to the induced (meet-)semilattice order on $D(S)$) such that $es=s$; see \cite{Csemi}.  The {\em natural order} on $S$ if given by $s\omega t$ if and only if $s=D(s)t$, or equivalently, $s=et$ for some $e\in D(S)$; the natural order is a compatible partial order on $S$ which restricts to the semilattice order on $D(S)$.   

Perhaps the most important example of a left restriction semigroup is the semigroup under composition of partial functions $X\rightarrow X$ for some non-empty set $X$, denoted $PT_X$ and equipped with domain operation given by $D(s)=\{(x,x)\mid x\in \dom(s)\}$.  Every left restriction semigroup embeds in $PT_X$ for some $X$, as was first shown in \cite{trok}; see also \cite{Csemi} and \cite{manes}.  

We say the left restriction semigroup $S$ has a {\em zero} if $S$ has a semigroup zero $0$ such that $0\in D(S)$ (equivalently, $D(0)=0$).  For example, $PT_X$ is a left restriction semigroup with zero the empty function.  We call any left restriction semigroup $S$ that is a monoid with identity $1$ a {\em left restriction monoid}; in this case, $1=D(1)1=D(1)$, so $1\in D(S)$, and the subset $S_1=\{s\in S\mid D(s)=1\}$ is a submonoid of $S$ (since if $s,t\in S_1$ then $D(st)=D(sD(t))=D(s1)=D(s)=1$).  The monoid, $(PT_X)_1$, consisting of the total functions on $X$, is identifiable with $T_X$, and we so identify them in what follows.

The left restriction semigroup axioms have been weakened in various ways.  The class of {\em LC-semigroups} is defined in \cite{Csemi} to consist of those unary semigroups $S$ with unary operation $D$ such that $D(S)=\{D(s)\mid s\in S\}$ is a semilattice, with $D(s)$ the smallest $e\in D(S)$ such that $es=s$.  These are given an equational axiomatisation in  \cite{Csemi}, generalising the above axioms for left restriction semigroups; indeed an LC-semigroup is a left restriction semigroup if and only if law \ref{r4} above holds.  The same concept was also considered by Batbedat in \cite{batbedat} where they were called type SL $\gamma$-semigroups.  There is an obvious dual concept of RC-semigroup defined in terms of a range operation $R$.

\subsection{Constellations, left restriction semigroups and demigroups}  \label{clrs}

Next we define the partial algebras that will form the basis of our main construction.  These are one-sided generalisations of categories called constellations.   

First, we say $e$ is a {\em right identity} for a partial binary operation $\circ$ if for all $a$, if $a\circ e$ exists then it equals $a$.
Following \cite{constgen}, a constellation $P$ is a partial algebra equipped with a partial binary operation and a unary ``domain" operation $D$, such that for all $x,y,z\in P$:
\begin{enumerate}[label=\textup{(C\arabic*)}]
\item if $x\circ(y\circ z)$ exists then so does $(x\circ y)\circ z$, and then the two are equal;  \label{c1}
\item if $x\circ y$ and $y\circ z$ exist then $x\circ(y\circ z)$ exists; \label{c2}
\item $D(x)$ is the unique right identity such that $D(x)\circ x=x$.  \label{c3} 
\end{enumerate}

These axioms generalise the object-free axioms for small categories: there is no range operation $R$, there is asymmetry in Law \ref{c1}, and one law for categories is missing entirely.

Let $P$ be a constellation.  Then $D(P)=\{D(s)\mid s\in P\}$ is the set of all right identities in $P$, and $e\circ e$ exists (hence equals $e$) for all $e\in D(P)$; moreover, for $s,t\in P$, $s\circ t$ exists if and only if $s\circ D(t)$ does.  For $e,f\in D(P)$, setting $e\leq f$ if and only if $e\circ f$ exists makes $\leq$ a quasiorder.  This extends to the {\em natural quasiorder} on $P$ given by $s\leq t$ if and only if $D(s)\circ t$ exists and equals $s$, or equivalently, $s=e\circ t$ for some $e\in D(P)$.  If one of these quasiorders is a partial order, so is the other, and then following \cite{constgen}, $P$ is said to be {\em normal}, and $\leq$ is called the {\em natural order} on $P$.  Normality of $P$ is evidently equivalent to the condition that for all $e,f\in D(P)$, if $e\circ f$ and $f\circ e$ exist then $e=f$. 

As in \cite{constgen}, a subset $Q$ of the constellation $P$ is a {\em subconstellation} of $P$ if $D(s)\in Q$ for all $s\in Q$, and if $s\circ t$ exists in $P$ for some $s,t\in Q$, then $s\circ t\in Q$.  In this case $Q$ is itself a constellation under the inherited partial operations, by Proposition $2.18$ in \cite{constgen}.  

There is a notion analogous to that of a functor for constellations.  Suppose $P,Q$ are constellations.  Following \cite{inductive}, we say the the function $\rho:P\rightarrow Q$ is a {\em radiant} if the following two conditions hold:
\ben
\item for all $s,t\in P$, if $s\circ t$ exists, then so does $(s\rho)\circ(t\rho)$, and $(s\circ t)\rho=(s\rho)\circ(t\rho)$, and
\item for all $s\in P$, $D(s)\rho=D(s\rho)$.
\een
The radiant $\rho$ is {\em strong} if for all $s,t\in P$, if $(s\rho)\circ(t\rho)$ exists then so does $s\circ t$, is an {\em embedding} if it is an injective strong radiant, and is an {\em isomorphism} if it is a surjective embedding in which case we write $P\cong Q$.  These definitions are the standard ones for partial algebras, as in \cite{gratzer} for instance.  It also follows that a radiant $\rho$ is an isomorphism if and only if it has an inverse function which is a radiant.

Constellations were introduced in \cite{inductive} in order to obtain a kind of ``ESN Theorem" for left restriction semigroups.  There, it was shown that a left restriction semigroup can be made into a constellation by defining $s\circ t:=st$ but only when $sD(t)=s$, and by retaining $D$.  When this is done, the natural order in the left restriction semigroup coincides with the natural order on the derived constellation.  A constellation arises from a left restriction semigroup in this way if and only if it is inductive in the sense of \cite{inductive}.  The inductive property was shown in \cite{constgen} to be equivalent to the following conditions on a constellation $P$.  
\bi
\item[\textup{(O4)}] If $e\in D(P)$ and $a\in P$, then there is a maximum $x\in P$ with respect to the natural quasiorder $\leq$ on $P$,  such that $x\leq a$ and $x\circ e$ exists, the {\em co-restriction} of $a$ to $e$, denoted $a|e$; and  \label{o4}
\item[\textup{(O5)}] for $x,y\in P$ and $e\in D(P)$, if $x\circ y$ exists then $D((x\circ y)|e)=D(x|(D(y|e))$. \label{o5}
\ei
Every inductive constellation is normal.  Note that the co-restriction $a|e$ in (O4) is fully determined (if it exists) by the natural order $\leq$ on the constellation $P$. 

A recent characterisation of the inductive property appearing in \cite{demonize} is convenient in what follows.  

\begin{lem}  \label{newindchar}
	The constellation $(P,\circ,D)$ is inductive if and only if it is normal and for all $s\in P$ and $e\in D(P)$, there exists $s\cdot e\in D(P)$ such that for all $t\in P$, $(t\circ s)\circ e$ exists if and only if $t\circ (s\cdot e)$ exists.  In that case, the co-restriction $s|e$ is $(s\cdot e)\circ s$ for all $s\in P$ and $e\in D(P)$.
\end{lem}

When a left restriction semigroup $S$ is viewed as a constellation $(S,\circ,D)$ as above, then the latter is inductive, and for $s,t\in S$, $s\circ t=st$ whenever the former is defined, and $s|e=se$ for all $e\in D(S)$. Conversely, it was shown in Proposition 4.5 of \cite{inductive} that an inductive constellation $P$ can be made into a left restriction semigroup by extending the partial multiplication to an everywhere-defined one.

\begin{dfn}  \label{pseudodef}
Let $P$ be an inductive constellation.  Define the {\em pseudoproduct} $\otimes$ on $P$ as follows: for all $s,t\in P$, $s\otimes t:= (s|D(t))\circ t$.  The resulting left restriction semigroup $(P,\otimes,D)$ is the {\em induced left restriction semigroup} on $P$.  
\end{dfn}

Note that in an inductive constellation $P$, we may also write $s\otimes t=((s\cdot D(t))\circ s)\circ t$, where $s\cdot D(t)$ is as in Lemma \ref{newindchar}.

It was shown in \cite{inductive} that the categories of inductive constellations (where the morphisms are ordered radiants as defined in \cite{inductive}) and left restriction semigroups (where the morphisms are semigroup homomorphisms respecting $D$) are isomorphic.  It is easily seen that the natural order in a left restriction semigroup is the same as the natural order when it is viewed as a constellation under this correspondence.

More generally, in a unary semigroup $S$ with unary operation $d$, one can define the {\em restricted product} $x\circ y:=xy$ only if $xd(y)=x$.  In \cite{demonize}, the unary semigroups $S$ for which $(S,\circ,d)$ is a constellation were characterised as the class of demigroups as in Definition \ref{demidef}.

\section{$E$-demigroups, their left $(d,E)$-completions and when they are inductive}  \label{extend}
 
In this section we formally show that $C^d_E(S)$ is a constellation, and characterise when such constellations are inductive and hence are ``left restriction semigroups in disguise".  This is preparation for a much closer examination of the two most important cases, namely when $S$ is a left $E$-monoid or left $E$-monoid with zero.  We start with properties of $C_E(S)$ where $S$ is a semigroup and $\emptyset\neq E\subseteq E(S)$, then extend them to $C^d_E(S)$. 

\begin{pro}  \label{iscon}
Suppose $S$ is a semigroup with $E$ a non-empty subset of its idempotents.  

The partial operation $\circ$ on $C_E(S)$ as in Definition \ref{CES} is well-defined, and $(C_E(S),\circ,D)$ is a constellation $P$ in which $D(P)$ is $\{(e,e)\mid e\in E\}$.  Moreover in $P$, $(e,s)\leq (f,t)$ under the natural quasiorder in $P$ if and only if $ef=e$ and $s=et$.   

If $S$ is an $E$-demigroup, then $C^d_E(S)$ is a subconstellation of $C_E(S)$.  
\end{pro}
\begin{proof}
Defining $(e,s)\circ (f,t)=(e,st)$ on $E\times S$ whenever $sf=s$, the proofs of \ref{c1} and \ref{c2} in Section \ref{clrs} are routine.  Next note that for $(e,s),(f,t)\in C_E(S)$ for which $(e,s)\circ (f,t)$ exists, we have $es=s$, $ft=t$, and $sf=s$.  Then $est=st$ and so $(e,st)\in C_E(S)$.  So $\circ$ is well-defined on $C_E(S)$ and so the properties \ref{c1} and \ref{c2} are inherited from $E\times S$.  

For $e\in E$, $ee=e$ so $(e,e)\in C_E(S)$, and $C_E(S)$ is closed under $D$.    
To show \ref{c3}, we first determine the right identities in $C_E(S)$.  As noted above, $(f,f)\in C_E(S)$ for all $f\in E$.
Now if $(e,s)\circ (f,f)$ exists for some $e,f\in E$ and $s\in S$ then $sf=s$, and so $(e,s)\circ (f,f)=(e,s)$, so $(f,f)$ is a right identity.  Conversely, if $(f,x)$ is a right identity, then $(f,f)=(f,f)\circ(f,x)=(f,x)$, giving $f=x$.  So the set of right identities in $C_E(S)$ is the set of elements of the form $(e,e)$ where $e\in E$.  

Now for all $(e,s)\in C_E(S)$, $(e,e)\circ (e,s)=(e,s)$, and if also $(f,f)\circ(e,s)=(e,s)$ then $f=e$, so $(f,f)=(e,e)$.  So \ref{c3} holds, and $C_E(S)$ is a constellation.

Suppose $(e,s)\leq (f,t)$ under the natural order on $P=C_E(S)$.  Then 
$$(e,s)=D((e,s))\circ (f,t)=(e,e)\circ (f,t)=(e,et),$$
so $ef=e$ and $s=et$.  Conversely, suppose $(e,s),(f,t)\in P$ are such that $ef=e$ and $et=s$.  Then $(e,e)\circ (f,t)$ exists and equals $(e,et)=(e,s)$, so $(e,s)\leq (f,t)$. 

Finally, suppose $S$ is an $E$-demigroup and consider $(e,s),(f,t)\in C^d_E(S)$ such that $(e,s)\circ (f,t)$ exists.  Then in $C_E(S)$, $(e,s)\circ (f,t)=(e,st)$ and indeed $d(st)=d(sd(t))=d(sd(f))=d(sf)=d(s)=d(e)$, so $(e,st)\in C^d_E(S)$.  Trivially $(e,e)\in C^d_E(S)$ for all $e\in E$.  So $C^d_E(S)$ is a subconstellation of $C_E(S)$.
\end{proof}

We note that a similar though dual construction to $C_E(S)$ is used in Section 3 of \cite{constgen} to convert a constellation into a category; much earlier, a similar dual construction was used in Section 9 of \cite{agree} to convert an RC-semigroup satisfying the congruence condition (that is, a right Ehresmann semigroup) into a category.  

In the introduction, we defined left $E$-monoids and left $E$-monoids with zero as special cases of $E$-demigroups and described how they all arose.  We now justify the claims made there.

\begin{pro}  \label{01}
Suppose $S$ is a monoid with $1\in E\subseteq E(S)$.  
\ben
\item If we define $d(s)=1$ for all $s\in S$, then $S$ is a left $E$-monoid.
\item If $S$ has zero and is integral, meaning that it has no zero divisors (so $S\backslash \{0\}$ is a subsemigroup), with $0,1\in E\subseteq E(S)$ and we define $d(0)=0$ with $d(s)\neq 0$ if $s\neq 0$, then $S$ is a left $E$-monoid with zero.  Moreover every left $E$-monoid with zero arises in this way.
\een
\end{pro}

\begin{proof}
(1) is immediate.  For (2), if $S$ is integral, $0,1\in E\subseteq E(S)$ and we define $d$ as above, an easy case analysis verifies the left $E$-monoid with zero properties.  Conversely, if $S$ is a left $E$-monoid with zero, we must only show that $S$ is integral.  But if $st=0$ then $0=d(st)=d(sd(t))$, so $sd(t)=0$; if $t\neq 0$ then $d(t)=1$ so $s=0$.
\end{proof}

If $S$ is an $E$-demigroup, the normality of $C^d_E(S)$ depends only on the domain elements $D(C_E(S))=\{(e,e)\mid e\in E\}=D(C^d_E(S))$.  The condition that $(e,e)\circ (f,f)$ and $(f,f)\circ (e,e)$ are both defined if and only if $(e,e)=(f,f)$ is equivalent to the condition that $ef=e$ and $fe=f$ imply $e=f$, and so we obtain the following.

\begin{pro}  \label{norpre}
For an $E$-demigroup $S$, $C^d_E(S)$ is normal if and only if $E$ is right pre-reduced.
\end{pro}

An $E$-demigroup arises from any demigroup $S$, if we set $E=d(S)=\{d(s)\mid s\in S\}$.  The membership condition $(e,s)\in C^d_E(S)$ is $es=s$ and $d(e)=d(s)$, that is, $e=d(s)$ (since for all $t\in S$, $d(t)=d(d(t)t)=d(d(t)d(t))=d(d(t))$ so $d$ is idempotent, and so $d(e)=e$ for all $e\in d(S)$).  Hence, $C^d_E(S)=\{(d(s),s)\mid s\in S\}$, and $(d(s),s)\circ (d(t),t)$ exists if and only if $sd(t)=s$, in which case it equals $(d(s),st)=(d(st),st)$.  This is evidently isomorphic to the constellation product $\circ$ obtained from the demigroup $S$ as in \cite{demonize}: that $s\circ t=st$ should exist if and only if $sd(t)=s$.

\begin{pro}  \label{isomconst} 
Suppose $S$ is an $E$-demigroup, with $E,E'$ two right-equivalent right pre-reduced subsets of $E(S)$ (with for all $e\in E$, $e\, {\mc L}\, e'\in E'$).  Suppose $d(S)\subseteq E'$ and $d(e)=d(e')$ for all $e\in E$. 
Then $S$ is an $E'$-demigroup with respect to the same demigroup operation $d$, and the normal constellations $P=C^d_E(S)$ and $P'=C^d_{E'}(S)$ are isomorphic.  In particular, if $S$ is an $E$-monoid, then $C_E(S)\cong C_{E'}(S)$, and if $S$ is an $E$-monoid with zero, then $C^0_E(S)\cong C^0_{E'}(S)$
\end{pro}
\begin{proof}
To show that $S$ is an $E'$-demigroup, note that for all $e'\in E'$, $e'd(e')=e'ed(e)=e'e=e'$.  Now define $\rho:P\rightarrow P'$ by setting $(e,s)\rho=(e',e's)$.  This is well-defined since $e'(e's)=e's$, and $d(e's)=d(e'd(s))=d(e'd(e))=d(e'e)=d(e')$.  Evidently, $\rho$ has inverse $\rho':P'\rightarrow P$ given by $(e',s)\rho'=(e,es)$, and each is a radiant, hence an isomorphism, as follows using Lemma \ref{bighelp} several times.  For if $(e,s)\circ (f,t)$ exists in $P$ then $sf=s$, so $sf'=s$ and so $(e,s)\rho\circ (f,t)\rho=(e',e's)\circ(f',f't)$ exists in $P'$, and equals $(e',e'sf't)=(e',e'st)=(e,st)\rho=((e,s)\circ (f,t))\rho$.  Moreover, $D((e,s))\rho=(e,e)\rho=(e',e')=D((e',e's))=D((e,s)\rho)$.
\end{proof}  

In \cite{Lawson}, Lawson showed that it is possible to extend the multiplication in certain types of small categories so that it is defined universally, giving a certain class of bi-unary semigroups: these categories he called Ehresmann.  Lawson was able to establish an isomorphism between such small categories and what he dubbed Ehresmann semigroups, referred to in Section \ref{leftrestetc}, which specialised to an ismorphism between inductive categories and two-sided restiction semigroups.  Natural examples of Ehresmann semigroups include algebras of binary relations under composition, domain and range.

From the point of view of semigroup theory, an interesting question therefore arises: for a semigroup $S$ and set of idempotents $E$, when is $Cat_E(S)$ inductive?  For in such cases, it can be turned into an Ehresmann semigroup.  In fact, this problem does not seem to have an easy answer.  The elements of an inductive category are ordered in two ways, and there is no obvious way to impose a partial order on $Cat_E(S)$. 

However, for an $E$-demigroup $S$, $C^d_E(S)$ is a constellation, and hence has its natural quasiorder.  So it is possible to characterise those choices of $E$-demigroup $S$ for which $C^d_E(S)$ is an inductive constellation. 

\begin{pro}  \label{converse}
	Suppose $S$ is an $E$-demigroup.  If $P=C^d_E(S)$ is inductive, then $E$ is right pre-reduced, and 
	\ben[label=\textup{(I\arabic*)}]
	\item for all $t\in S$ and $e\in E$, there is $t\cdot e\in E$ such that for all $s\in S$: ($ste=st$ and $sd(t)=s$) if and only if $s(t\cdot e)=s$, and  \label{i1}
	\item $(E,\omega^l)$ is a meet-semilattice, and for all $s\in S$ and $e,f\in E$, 
\[se=sf=s \Rightarrow s(e\wedge f)=s. \label{i2}\]
	\een
\end{pro}
\begin{proof}
Suppose $P=C^d_E(S)$ is inductive.  Then $P$ is normal so $E$ is right pre-reduced by Proposition \ref{norpre}.  

Pick $t\in S$ and $e\in E$.  Then $(d(t),t)\in P$ since $d$ is idempotent.  By Lemma \ref{newindchar}, there is $(d(t),t)\cdot(e,e)\in D(P)$, which we denote by $(t\cdot e,t\cdot e)$ for some $t\cdot e\in E$, such that for all $(f,s)\in P$,  $((f,s)\circ (d(t),t))\circ(e,e)$ exists if and only if $(f,s)\circ (t\cdot e,t\cdot e)$ exists.  In particular, picking $s\in P$ and letting $f=d(s)$, we have that $((d(s),s)\circ (d(t),t))\circ(e,e)$ exists if and only if $(d(s),s)\circ (t\cdot e,t\cdot e)$ exists, which is to say that $sd(t)=s$ and $ste=st$ if and only if $s(t\cdot e)=s$.
	
For $e,f\in E$, it follows from Proposition \ref{iscon} that $(e,e)\leq (f,f)$ in $D(C_E(S))$ if and only if $ef=e$ in $E$, that is, $e\ \omega^l\ f$ in $E$.  So as a poset under $\omega^l$, $E$ is isomorphic to $D(C_E(S))=D(P)$ under its natural order.  In particular, $(E,\omega^l)$ is a meet-semilattice, and working in the induced left restriction semigroup on $P$ as in Definition \ref{pseudodef}, $(e,e)\otimes (f,f)=(e,e)\wedge (f,f)=(e\wedge f,e\wedge f)$ where ``$\wedge$" is interpreted appropriately in each case.

Next suppose $se=s=sf$ for $s\in S$ and $e,f\in E$.  Then $$(d(s),s)\circ(e,e)=(d(s),s)=(d(s),s)\circ(f,f),$$ so in the induced left restriction semigroup on $P$, $$(d(s),s)=(d(s),s)\otimes (e,e)=(d(s),s)\otimes (f,f)=(d(s),s)\otimes (e,e)\otimes (f,f)=(d(s),s)\otimes (e\wedge f,e\wedge f).$$  
Hence, $(d(s),s)\circ (e\wedge f,e\wedge f)=(d(s),s)$, and so $s(e\wedge f)=s$.
\end{proof}

We use the above result to define the notion of an inductive $E$-demigroup.

\begin{dfn}  \label{inddef}
Suppose $S$ is an $E$-demigroup.  We say $S$ is {\em inductive} if $E$ is right pre-reduced and there is a function $\cdot: S\times E\rightarrow E$ satisfying \ref{i1} and \ref{i2} in Proposition \ref{converse} above.
\end{dfn}  

Proposition \ref{converse} may be rephrased by saying that if $S$ is an $E$-demigroup and $C^d_E(S)$ is inductive then $S$ is inductive.  We now prove its converse.

\begin{pro}  \label{isind}
If $S$ is an inductive $E$-demigroup, then $C^d_E(S)$ is an inductive constellation in which 
$$(f,s)|(e,e)=(f\wedge (s\cdot e),(f\wedge (s\cdot e))s).$$
\end{pro}
\begin{proof} 
Pick $(f,s),(e,e)\in C^d_E(S)$ and let $h=f\wedge (s\cdot e)$.
Then for all $(g,t)\in C^d_E(S)$, the following are equivalent:
\bi
\item the constellation product $((g,t)\circ (f,s))\circ (e,e)$ exists;
\item $tf=t$ and $tse=ts$;
\item $tf=t$, $td(s)=t$ and $tse=ts$;
\item $tf=t$ and $t(s\cdot e)=t$;
\item $t(f\wedge (s\cdot e))=t$;
\item the constellation product $(g,t)\circ (h,h)$ exists.
\ei
The second line above implies the third because, given that $d(f)=d(s)$ (since $(f,s)\in C^d_E(S)$), if $tf=t$ then $td(s)=tfd(s)=tfd(f)=tf=t$. So by Lemma \ref{newindchar}, $C_E(S)$ is inductive, and $(f,s)|(e,e)=D((f,s)|(e,e))\circ (f,s)=(h,h)\circ (f,s)=(h,hs)$.
\end{proof}

\begin{cor}  \label{modrestcor}
Let $S$ be an inductive $E$-demigroup, and let $P=C^d_E(S)$.  Then $P$ can be made into a left restriction semigroup by defining $(e,s)(f,t)=(e\wedge (s\cdot f),(e\wedge (s\cdot f))st)$ for all $(e,s),(f,t)\in P$, and $D((e,s))=(e,e)$.  Moreover $(e,s)\leq(f,t)$ under the natural order if and only if $e\ \omega^l\ f$ and $s=et$. 
\end{cor}
\begin{proof}
As in Definition \ref{pseudodef}, multiplication in the induced left restriction semigroup of $P$ is given by $(e,s)\otimes (f,t)=((e,s)|(f,f))\circ (f,t)=(e\wedge(s\cdot f),(e\wedge(s\cdot f))s)\circ (f,t)=(e\wedge(s\cdot f),(e\wedge(s\cdot f))st)$, which agrees with the definition in the corollary statement; hence $P$ is a left restriction semigroup under the given operations.  By Proposition \ref{iscon}, the natural order is as stated.
\end{proof}

\begin{dfn}  \label{restdefd} 
Let $S$ be an inductive $E$-demigroup.  Denote by $Rest_d(E,S)$ the induced left restriction semigroup structure on $C^d_E(S)$ as in Corollary \ref{modrestcor}:
\[(e,s)(f,t)=(e\wedge (s\cdot f),(e\wedge (s\cdot f))st) \mbox{ and } D((e,s))=(e,e).\]
\end{dfn}

Note that for an inductive $E$-demigroup $S$, $Rest_d(E,S)$ has the same underlying set as $C^d_E(S)$, the unary domain operation is the same in both, and multiplication on $Rest_d(E,S)$ agrees with that on $C^d_E(S)$ whenever the latter is defined.  Moreover the extension of the partial multiplication on $C^d_E(S)$ to that on $Rest_d(E,S)$ is entirely determined by the structure of $C^d_E(S)$ as a constellation.  By Proposition \ref{isomconst}, we therefore obtain the following.

\begin{cor} \label{isomrest}
If $S$ is an inductive $E$-demigroup, and $E'\sim_l E$ with $E\ni e\, {\mc L}\, e'\in E'$, and $d(e')=d(e)$ for all $e\in E$, then $S$ is also an inductive $E'$-demigroup.  Moreover, $Rest_d(E,S)\cong Rest_d(E',S)$.
\end{cor}

In \cite{demonize}, it was shown how to view a demigroup $S$ as a constellation by defining $s\cdot t$ only if $sd(t)=s$ and retaining $d$ as the domain operation; indeed demigroups were shown to be the most general class of unary algebras for which this partial operation and $d$ together give a constellation.  It is easily seen that the constellation $(S,\cdot,d)$ is isomorphic to $C^d_{d(S)}(S)$ under the correspondence $s\leftrightarrow (d(s),s)$.  Theorem $5.1$ in \cite{demonize} characterises when $(S,\cdot,d)$ is inductive in terms of the demigroup $S$; such demigroups were called inductive in \cite{demonize}.  Because of the isomorphism between $(S,\cdot,d)$ and $C^d_{d(S)}(S)$, this characterisation arises as a special case of Propositions \ref{converse} and \ref{isind}.  But, in what follows, the most important special cases arise from left $E$-monoids and left $E$-monoids with zero.

Viewing a left $E$-monoid $S$ as an $E$-demigroup in the usual way (by setting $d(s)=1$ for all $s\in S$), the following is immediate from Propositions \ref{converse} and \ref{isind}.

\begin{cor} \label{Emodcor}
Let $S$ be a left $E$-monoid.  Then  $S$ is inductive if and only if $E$ is right pre-reduced, and the following hold:
	\ben[label=\textup{(I\arabic*$^\prime$)}]
	\item for all $t\in S$ and $e\in E$, there is $t\cdot e\in E$ such that for all $s\in S$, $ste=st$ if and only if $s(t\cdot e)=s$, and  \label{i1'}
	\item $(E,\omega^l)$ is a meet-semilattice, such that all $s\in S$ and $e,f\in E$, 
\[se=sf=s \Rightarrow s(e\wedge f)=s. \label{i2'}\]
	\een
\end{cor}

Note that \ref{i1'} simply asserts that $(te,t)$ has a (necessarily unique) left equalizer in $E$.  

\begin{pro} \label{0agrees}
Let $S$ be a left $E$-monoid with zero (so that $d(0)=0$ and $d(s)=1$ if $s\neq 0$), and let $S'$ be the left $E$-monoid obtained from $S$ by viewing it as an $E$-demigroup as in the first part of Proposition \ref{01} (by setting $d'(s)=1$ for all $s$).  Then $S$ is inductive if and only if $S'$ is, and the multiplication in the induced left restriction semigroup of $C^0_E(S)$ agrees with that in $C_E(S')$.
\end{pro}
\begin{proof}
Suppose $S$ is inductive. Then \ref{i1} and \ref{i2} in Proposition \ref{converse} are satisfied.  Pick $t\neq 0$ and $e\in E$; then $d(t)=1$, and \ref{i1} in Proposition \ref{converse} asserts that $(te,t)$ has an equalizer $t\cdot e$ in $E$.  Now let $t=0$ so that $d(t)=0$, and then we see that $(t,te)=(0,0)$ have equalizer $1\in E$.  So \ref{i1'} in Corollary \ref{Emodcor} is satisfied, and \ref{i2'} is the same as \ref{i2}, and so $S'$ is inductive as well.  

Conversely, suppose $S'$ is inductive.  Then \ref{i1'} and \ref{i2'} in Corollary \ref{Emodcor} are satisfied.  Again, for $t\neq 0$, $d(t)=d'(t)=1$ and \ref{i1'} in Corollary \ref{Emodcor} immediately gives \ref{i1} in Proposition \ref{converse}.  For $t=0$, $d(0)=0$ and then \ref{i1} in Proposition \ref{converse} asserts that for all $s$, $0=0$ and $s0=s$ if and only if $s(0\cdot e)=s$, which is true if $0\cdot e$ is chosen to be $0$.   So whether or not $t=0$, \ref{i1} is satisfied by the $E$-demigroup $S$. Again, \ref{i2} is the same as \ref{i2'}.  So both \ref{i1} and \ref{i2} in Proposition \ref{converse} are satisfied and so $S$ is inductive.

In the above, note that the choice of equalizer of $(t,te)$ is the same in either $S$ or $S'$, providing $t\neq 0$.  Consider $(e,s),(f,t)\in C^0_E(S')$.  Then the product $(e,s)(f,t)=(e\wedge (s\cdot f),(e\wedge (s\cdot f))st)$, which agrees with the product calculated in $C_E(S)$ providing $s\neq 0$.  But if $s=0$ then $e=0$ and so we obtain $(0,0)$ for this product whether calculated in $C^0_E(S')$ or $C_E(S)$.
\end{proof}

It follows that for a left $E$-monoid with zero $S$, precisely the same conditions characterise the inductive property regardless of whather $S$ is viewed as a left $E$-monoid with zero or as a left $E$-monoid.

\begin{cor} \label{Emodcor2}
Let $S$ be a left $E$-monoid with zero.  Then $S$ is inductive if and only if $E$ is right pre-reduced, and the following hold:
	\ben[label=\textup{(I\arabic*$^\prime$)}]
	\item for all $t\in S$ and $e\in E$, there is $t\cdot e\in E$ such that for all $s\in S$, $ste=st$ if and only if $s(t\cdot e)=s$, and  
	\item $(E,\omega^l)$ is a meet-semilattice, and for all $s\in S$ and $e,f\in E$, 
\[se=sf=s \Rightarrow s(e\wedge f)=s.\]
	\een
\end{cor}

Next we specialise Definition \ref{restdefd} to our two important special cases, and spell out exactly what the elements look like and how multiplication works.

\begin{dfn}  \label{restdef} 
Let $S$ be an inductive left $E$-monoid, so that $d(s)=1$ for all $s\in S$.  Write $Rest(E,S)$ rather than $Rest_d(E,S)$, the {\em semigroup left $E$-completion of $S$}.  Thus,
$$Rest(E,S)=\{(e,s)\mid e\in E, s\in S,es=s\},$$
with multiplication given by $$(e,s)(f,t)=(e\wedge (s\cdot f), (e\wedge (s\cdot f))st),$$ 
and $D((e,s))=(e,e)$.

If $S$ is integral, viewing it as a left $E$-monoid with zero as in the second part of Proposition \ref{01}, write $Rest_0(E,S)$ rather than $Rest_d(E,S)$, the {\em zero-reduced semigroup left $E$-completion of $S$}.  Thus,  
$$Rest_0(E,S)=\{(e,s)\in Rest(E,S) \mid e=0\Rightarrow s=0\},$$
with multiplication and $D$ as defined for $Rest(E,S)$ above.

If $S$ is an inductive right $E$-monoid, denote by $RRest(S,E)$ and $RRest_0(S,E)$ the obvious dually defined right restriction monoids.  
\end{dfn}

If $S$ is an inductive $E$-monoid with zero, then by Proposition \ref{0agrees}, $Rest_0(E,S)$ is a subalgebra of the left restriction semigroup $Rest(E,S)$ viewed as a unary semigroup.

The following says that every element of $Rest(E,S)$ (hence also of $Rest_0(E,S)$ if $S$ is integral) may be factored into a product of a domain element and an element of domain $1$.

\begin{cor}  \label{factor}
Suppose $S$ is an inductive left $E$-monoid and $T=Rest(E,S)$.  Every element of $T$ may be written as $u=D(u)v$ for some $v\in T_1$; moreover if $u\in T$ and $u=gv$ for some $g\in D(T)$ and $v\in T_1$, then $g=D(u)$. 
\end{cor}
\begin{proof}
For $(e,s)\in T$, we have $(e,s)=(e,e)(1,s)$, and if $(e,s)=(f,f)(1,t)=(f,ft)$, then $e=f$, so $(f,f)=D((e,s))$.
\end{proof} 

The factorization in the above corollary is not in general unique, since, starting with $t\in S$, and selecting $e\in E$, we have that $(e,e)(1,t)=(e,et)=(e,e)(1,et)$, and $t,et\in S$ may be unequal.

\section{Characterising left restriction monoids that are semigroup left $E$-completions}  \label{lrchar}

Building on the results of Section \ref{extend}, we may ask which left restriction semigroups $M$ arise as $Rest_d(E,S)$ for some inductive $E$-demigroup $S$.  But the answer is trivial: all of them.  

Given a left restriction semigroup $(S,\times,D)$, letting $d=D$, it is immediate that $S$ is a demigroup; hence it determines a (necessarily inductive) constellation $(S,\cdot, d)$ under the restricted product $s\cdot t=st$ which only exists when $sd(t)=s$, as discussed in the comments following Corollary \ref{isomrest}.  We can also view it as a $d(S)$-demigroup and form $C^d_{d(S)}(S)$ which, as noted in those comments, is isomorphic to $(S,\cdot,d)$.  Since the latter is inductive, it has an induced left restriction semigroup structure $S_d$, called the demonization of $S$ in \cite{demonize}.  So the induced left restriction semigroup obtained from $C^d_{d(S)}(S)$, namely $Rest_d(d(S),S)$, must be isomorphic to $S_d$. As noted in \cite{demonize}, the demonization of a left restriction semigroup is nothing but a copy of that same left restriction semigroup, and so $S\cong Rest_d(d(S),S)$.

The situation is more complex and interesting if we take $S$ to be a left $E$-monoid, or a left $E$-monoid with zero and extend it via $Rest(E,S)$ or $Rest_0(E,S)$.  Moreover, several important examples have this form, as we see in Section \ref{egs}.

Recall that if $S$ is a left restriction monoid, $S_1$ is the submonoid $\{s\in S\mid D(s)=1\}$.  

\begin{dfn} 
Let $S$ be a left restriction monoid.  Then $f\in E(S)$ is a {\em large idempotent} if $f\in E(S_1)$.  We say that $S$ has {\em enough large idempotents} if for all non-zero $e\in D(S)$, there exists $f\in E(S_1)$ such that $e\, {\mc L}\, f$.  We say that $S$ has {\em precisely enough large idempotents} if it has enough large idempotents and for all $f\in E(S_1)$ there is $e\in D(S)$ such that $e\, {\mc L}\, f$.  Let $E\subseteq E(S)$ consist of exactly one $e'\in E(S_1)$ for each non-zero $e\in D(S)$ such that $e'\, {\mc L}\, e$, as well as $0\in D(S)$ if $S$ has a zero; then we say {\em $E$ is a set of enough large idempotents}.
\end{dfn}

Recall that the generalised Green's relation $\widetilde{\mc {R}}_E$ can be defined on any semigroup $S$, once $E\subseteq E(S)$ has been chosen: for all $s,t\in S$, $(s,t)\in \widetilde{\mc {R}}_E$ if and only if for all $e\in E$, $es=s$ if and only if $et=t$.  (There is a dually defined relation $\widetilde{\mc {L}}_E$.)  The relation $\widetilde{\mc {R}}_E$ is asociated with the one-sided version of $E$-semiabundance, considered by various authors.  If $S$ is a left restriction monoid and we choose $E=D(S)$, then $(s,t)\in \widetilde{\mc {R}}_E$ if and only if $D(s)=D(t)$.  It follows easily that another way to say that $S$ has enough large idempotents is to say that for every non-zero $e\in D(S)$, there is an idempotent $f$ such that $e\ {\mc L}\ f\ \widetilde{\mc {R}}_{D(S)}\ 1$.  

\begin{eg}  A left restriction semigroup with enough large idempotents.  \label{enougheg} \end{eg}
Let $X=\{x,y\}$ be a two-element set.  Then $S=PT_X=\{1,0,p_x,p_y,e,f,i,j,k\}$, where $1$ is the identity map, $0$ is the empty function, $p_x=\{(x,x),(y,x)\}$, $p_y=\{(x,y),(y,y)\}$, $e_x=\{(x,x)\}$, $e_y=\{(y,y)\}$, $i=\{(x,y),(y,x)\}$, $j=\{(x,y)\}$, $k=\{(y,x)\}$.  Then $E=D(S)=\{1,0,e_x,e_y\}$, $E(S)=D(S)\cup \{p_x,p_y\}$, and $S_1=\{1,p_x,p_y,i\}$, which is of course equal to $T_X$.  Note that the idempotents in the ${\mc L}$-classes are $\{0\},\{1\},\{p_x,e_x\}$ and $\{p_y,e_y\}$, and the idempotents in the $\tilde{\mc {R}}_E$-classes are $\{0\}$, $\{e_y\}$ and $\{1,p_x,p_y\}$.  We consider an ``egg-box diagram" for $E(S)$, but one using ${\mc L}$ for the columns and $\tilde{\mc {R}}_E$ for the rows, rather than ${\mc L}$ and ${\mc R}$.
\bigskip

\centerline{\begin{tabular}{|c|c|c|c|}
\hline
0&&&\\
\hline
&1&$p_x$&$p_y$\\
\hline
&&$e_x$&\\
\hline
&&&$e_y$\\
\hline
\end{tabular}}
\bigskip

Evidently, for each non-zero element of $D(S)$ (that is, one of $1,e_x,e_y$), there is an element from the row containing $1$ in its column, so $S$ has enough large idempotents.  Indeed, in this case it has precisely enough large idempotents, since for each idempotent in the row containing $1$, there is an element of $D(S)$ in its column.  We shall see in Subsection \ref{pareg} that this observation about $PT_X$ generalises.

\begin{pro}  \label{key}
Suppose $S$ is an inductive left $E$-monoid.  Then $M=Rest(E,S)$ has enough large idempotents.  If $S$ is integral then $M=Rest_0(E,S)$ has enough large idempotents.
\end{pro}
\begin{proof}  For $e\in E$, note that $e\cdot e=1$.  A typical non-zero element of $D(M)$ is $(e,e)$ where $e\neq 0$.  Then $(1,e)\in M$, and $(1,e)(1,e)=(1,e)\circ(1,e)=(1,ee)=(1,e)$, so $(1,e)\in E(M_1)$, similarly $(e,e)(1,e)=(e,ee)=(e,e)$, and 
$(1,e)(e,e)=(1,ee)=(1,e)$, so $(e,e)\, {\mc L}\, (1,e)$.  So $M$ has enough large idempotents.
\end{proof} 

Note that if $S$ is an inductive left $E$-monoid, then in $M=Rest(E,S)$, $M_1\cong S$ as monoids, under the mapping $(1,s)\mapsto s$ for all $s\in S$, while if $S$ is integral, then in $M=Rest_0(E,S)$, $M_1\cup \{0\}\cong S$ under the mapping $(1,s)\mapsto s$ for all non-zero $s\in S$ with $(0,0)\mapsto 0$.   Moreover, $M$ has enough large idempotents by Proposition \ref{key}.

Conversely, we have the following two results. 

\begin{thm}  \label{char} 
Let $S$ be a left restriction monoid without zero and with enough large idempotents.  For each $e\in D(S)$, let $e'\in E(S_1)$ be such that $e\, {\mc L}\, e'$, and let $E=\{e'\mid e\in D(S)\}$.  Then $S_1$ is an inductive left $E$-monoid in which, if $t\in S_1$ then $t\cdot e'=D(te)'$ for all $e'\in E$, and $S\cong Rest(E,S_1)$.
\end{thm}
\begin{proof} 
First note that $E$ is right pre-reduced, as if $e'\, {\mc L}\, f'$ then $e\, {\mc L}\, f$ and so $e=f$, so $e'=f'$.  Moreover, if $e\, {\mc L}\, 1$ for some $e\in E$, then $e=1$, so $1\in E$, which corresponds to $1\in D(S)$.  So $S_1$ is an $E$-monoid.  Note also that $E$ is a meet-semilattice under $\omega^l$ since $E$ and $D(S)$ are isomorphic as posets under $\omega^l$ as discussed prior to Lemma \ref{bighelp}.  

We first show that $S_1$ satisfies \ref{i1'} in Proposition \ref{Emodcor}.  For $f'\in E$, we want to show that for all $s\in S_1$, $stf'=st$ if and only if $s=sD(tf)'$.  By the first part of Lemma \ref{bighelp}, it is necessary and sufficient to show that $stf=st$ if and only if $sD(tf)=s$.  Now if $stf=st$ then $sD(tf)=D(stf)s=D(st)s=D(sD(t))s=s$ since $D(t)=1$, while if $sD(tf)=s$ then $stf=sD(tf)t=st$. 

Now we show $S_1$ satisfies \ref{i2'} in Proposition \ref{Emodcor}.  We must show that for $s\in S_1$ and $e',f'\in E$, $s=se'=sf'$ imply $s=s(e'\wedge f')=s(ef)'$; again by the first part of Lemma \ref{bighelp}, it is necessary and sufficient to show that $se=sf=s$ implies $sef=s$, for all $s,t\in S_1$ and $e,f\in D(S)$.  But this is almost immediate.  
	
Let $T=Rest(E,S_1)$, as in Definition \ref{restdef} (see also Definition \ref{inddef}).  We claim that the map $\theta:T\rightarrow S$ given by $(e',s)\theta=es$ is an isomorphism.
	
First, suppose $(e',s),(f',t)\in T$.  If $(e',s)\theta=(f',t)\theta$ then $es=ft$, and by the left congruence property, $D(es)=D(eD(s))=D(e)=e$ (since $D(s)=1$), and similarly $D(ft)=f$, so $e=f$, and so $es=et$, so $s=e's=e't=t$ by the second part of part of Lemma \ref{bighelp}.   So $(e,s)=(f,t)$, and so $\theta$ is injective.  
		
Now pick $s\in S$.  Let $t=D(s)'s$, so $D(t)=D(D(s)'s)=D(D(s)'D(s))=D(D(s)')=1$, so $t\in S_1$, and 
$D(s)t=D(s)D(s)'s=D(s)s=s$, so $s\leq t$ in $S$.  Then $(D(s)',t)\theta=D(s)t=s$.  So $\theta$ is surjective, and is therefore a bijection.
	
Now suppose $(e',s),(f',t)\in T$.  Then $((e',s)(f',t))\theta=(e'\wedge(s\cdot f'),(e'\wedge(s\cdot f'))st)\theta=eD(sf)st=esft=(e',s)\theta(f',t)\theta$.  So $\theta$ is a semigroup homomorphism.
	
Finally, we must show $D$ is respected by $\theta$.  But $D((e',s))\theta=(e',e')\theta=ee'=e$, while $D((e',s)\theta)=D(es)=D(eD(s))=D(e1)=e$ also.
	
So $\theta$ is an isomorphism and $S\cong T$ as left restriction monoids.
\end{proof}

\begin{thm}  \label{char0}
	Let $S$ be a left restriction monoid with zero having enough large idempotents.  For each non-zero $e\in D(S)$, let $e'\in E(S_1)$ be such that $e\, {\mc L}\, e'$, and also let $0'=0$.  Let $E=\{e'\mid e\in D(S)\}$.  
Let $S'=S_1\cup\{0\}$.  Then $S'$ is an inductive left $E$-monoid with zero in which 
\bi 
\item if $t\in S'$ and $t\neq 0$ then $t\cdot e'=D(te)'$ for all $e'\in E$, and
\item if $S$ has zero then $0\cdot g=1$ for all $g\in E$,
\ei
and $S\cong Rest_0(E,S')$.
\end{thm}
\begin{proof} 
The arguments that $E$ is right pre-reduced, that $1\in E$ corresponds to $1\in D(S)$, and that $E$ is a meet-semilattice under $\omega^l$ are the same as that given in the proof of Theorem \ref{char}, with the case of $0$ easily accommodated.

We next show that $S'$ satisfies \ref{i1'} in Proposition \ref{Emodcor}.  If $t=0$, we must show that for any $g\in E$, $s0g=s0$ if and only if $s1=s$, which is trivially true.  If $t\in S'$ is non-zero, the argument is the same as that given in the proof of Theorem \ref{char} for $S_1$.  That $S'$ satisfies \ref{i2'} in Proposition \ref{Emodcor} follows the same argument as in the proof of Theorem \ref{char} for $S_1$. 

For $s,t\in S_1$, $D(st)=1$ since $S_1$ is a submonoid, so $st\neq 0$ and so $S'$ is integral.
	
Let $T=Rest_0(E,S')$, as in Definition \ref{restdef}. We claim that the map $\theta:T\rightarrow S$ given by $(e',s)\theta=es$ if $e'\neq 0$, with $(0,0)\theta=0$ if $S$ has zero, is an isomorphism.
	
First, suppose $(e',s),(f',t)\in T$.  If $M$ has zero and $(e',s)\neq (0,0)$, then $(e',s)\theta=es\neq 0$ since $S'$ is integral, so $(e',s)\theta\neq (0,0)\theta$.  If neither $(e',s)$ nor $(f',t)$ is $(0,0)$ and $(e',s)\theta=(f',t)\theta$, then the argument that $(e,s)=(f,t)$ is very similar to that for injectivity of $\theta$ given in the proof of Theorem \ref{char}.  Moreover neither $es$ nor $ft$ is zero.  So $\theta$ is injective.  
		
Now pick $s\in S$, $s\neq 0$.  Let $t=D(s)'s$, and then $(D(s)',t)\theta=D(s)t=s$ as in the proof of Theorem \ref{char}.  Furthermore, if $S$ has a zero then $(0,0)\theta=0$.  So $\theta$ is surjective, and is therefore a bijection.
	
Now suppose $(e',s),(f',t)\in T$.  If neither is $(0,0)$, then $((e',s)(f',t))\theta=(e',s)\theta(f',t)\theta$ as in the proof of Theorem \ref{char}.  If $S$ has zero, then because $(0,0)\theta=0$, it is immediate that $((0,0)(f,t))\theta=((e,s)(0,0))\theta=(0,0)\theta=0=(e,s)\theta(0,0)\theta=(0,0)\theta(f,t)\theta$.  So $\theta$ is a semigroup homomorphism.
	
Finally, we must show $D$ is respected by $\theta$.  The argument that $D((e',s))\theta=D((e',s)\theta)$ if $(e',s)\neq (0,0)$ is as in the proof of Theorem \ref{char}.  Otherwise, $D((0,0))\theta=(0,0)\theta=0=D(0)=D((0,0)\theta)$.
	
So $\theta$ is an isomorphism and $S\cong T$ as left restriction monoids.
\end{proof}

\begin{pro}  \label{max}
If $S$ is an inductive left $E$-monoid, then $E$ is maximal right pre-reduced if and only if $Rest(E,S)$ has precisely enough large idempotents (if and only if $Rest_0(E,S)$ has precisely enough large idempotents if $S$ is integral).  
\end{pro}
\begin{proof}
Suppose $S$ is an inductive left $E$-monoid.   

First assume that $E$ is maximal right pre-reduced.  For $(1,s)\in E(Rest(E,S)_1)$ ($=E(Rest_0(E,S)_1)$ if $S$ is integral), then $(1,s)(1,s)=(1\wedge (s\cdot 1),(1\wedge (s\cdot 1))s^2)$, so because $s\cdot 1=1$, we have that $s^2=s$ in $S$, and so $s\in E(S)$.  Because $E$ is maximal right pre-reduced in $E(S)$, there is $f\in E$ such that $f\, {\mc L}\, s$, so $fs=f$ and $sf=s$.  Hence, $(f,f)(1,s)=(f,fs)=(f,f)$ and $(1,s)(f,f)=(s\cdot f,(s\cdot f)sf)=(1,sf)=(1,s)$, so $Rest(E,S)$ (and $Rest_0(E,S)$ if $S$ has zero) has precisely enough large idempotents.  

Conversely, if $Rest(E,S)$ (or $Rest_0(E,S)$ if $S$ is integral) has precisely enough large idempotents, then for $(1,s)\in E(Rest(E,S)_1) (=E(Rest_0(E,S)_1)$ if $S$ has zero), $s\in E(S)$ and there exists (non-zero) $(e,e)\in D(Rest(E,S))$ such that $(e,e)(1,s)=(e,e)$ and $(1,s)(e,e)=(1,s)$, so $(e,es)=(e,e)$ and $(s\cdot e,(s\cdot e)se)=(1,s)$, so $es=e$, and $s\cdot e=1$ so $se=s$, and so $s\, {\mc L}\, e$, showing that $E$ is maximal right pre-reduced.
\end{proof}

\begin{cor}  \label{char02}
If a left restriction monoid $S$ has precisely enough large idempotents, then it is isomorphic to $Rest_0(E,S')$ if $S$ has zero or to $Rest(E,S')$ if not, where $E$ and $S'$ are as in Theorem \ref{char0} and $E$ is maximal right pre-reduced in $E(S)$.
\end{cor}

The results of this section show in particular that a left restriction semigroup $S$ with enough large idempotents (with or without zero) may be reconstructed from its submonoid $S_1$ (consisting of precisely those elements of $S$ that are $\tilde{\mc R}_{D(S)}$-related to $1$).  This may be done by selecting exactly one idempotent in $S_1$ (together with $0$ if $S$ has zero) from the ${\mc L}$-class of each member of $D(S)$ to form $E$, and then forming $Rest(E,S_1)$ (or $Rest_0(E,S')$ where $S'=S_1\cup\{0\}$ if $S$ has zero) as in the proof of Theorem \ref{char} (Theorem \ref{char0} if $S$ has zero).   If $S$ has precisely enough large idempotents, this can be viewed as selecting exactly one idempotent from each ${\mc L}$-class of $S_1$ that contains an idempotent (together with zero if $S$ has zero) and then forming $Rest(E,S_1)$ (or $Rest(E,S')$ if $S$ has zero).  In Section \ref{egs}, we shall see that this reconstruction technique applies to some of the most important examples of left restriction semigroups.

\section{Inductive left $E$-monoids in more detail}  \label{leftEmod}

Their signficance established, in this section we consider inductive left $E$-monoids for their own sake.  We show how in every monoid there is a single (modulo ${\mc L}$) largest choice of $E$ with respect to which it is an inductive left $E$-monoid.  But first, it proves useful to consider actions satisfying only the first condition of the inductive property for left $E$-monoids.  

\begin{dfn}  \label{modal}
Suppose $S$ is a left $E$-monoid.  Then $S$ is {\em modal} if :
\ben
\item $E$ is right pre-reduced, that is, for all $e,f\in E$, $e=ef$ and $f=fe$ imply $e=f$, and 
\item \ref{i1'} in Corollary \ref{Emodcor} holds, that is, for all $t\in S$, $e\in E$, $(te,t)$ have a (necessarily unique) left equalizer in $E$, denoted by $t\cdot e$.  
\een
When extended to all of $S\times E$, we call the induced mapping $(t,e)\mapsto t\cdot e$ the {\em left $E$-modal operation} on $S$.  
\end{dfn}

Note that the definition could instead have specified that for all $t\in S$, $e\in E$, $(te,t)$ have a unique left equalizer in $E$, and then the right pre-reduced property would have followed: for if $e\, {\mc L}\, f$ for some $e,f\in E$, then $Eq(1,1e)=Eq(1,e)=Se$, so $e$ equalizes $(1,1e)$, and since $e\, {\mc L}\, f$, so does $f$ by Proposition \ref{eq}, and so $e=f$ by uniqueness.

In general, simply defining $s\cdot e=D(se)$ for all $s\in S$ and $e\in D(S)$ in a left restriction monoid $S$ (which induces a familiar action of $S$ on $D(S)$) does not make it into a modal left $D(S)$-monoid: additional structure is needed. Nor does restricting to $S_1$ and making the same definition turn $S_1$ into a modal left $D(S)$-monoid, since, aside from $1\in S$, $D(S)$ lies outside $S_1$.   

The terminology ``modal" used in Definition \ref{modal} is in fact motivated from \cite{modrest}, since the so-called modal restriction semigroups considered there are special cases of modal left $E$-monoids.  A modal restriction semigroup is a left restriction monoid with zero, $S$, such that for every $s\in S$ there is a largest $e\in D(S)$ for which $es=0$ (call this $A(s)$), also satisfying certain other laws so that $A$ models domain complement for partial functions. It then follows that for all $s\in S$ and $e\in D(S)$, $A(sA(e))$ is the largest $f\in D(S)$ for which $fse=fs$, and so $S$ is a modal left $D(S)$-monoid.

\begin{pro}  \label{eqmod0cor}
If $S$ is a modal left $E$-monoid with $s\in S$ and $e\in E$, then for all $f\in E$, $f\in Eq(se,s)$ if and only if $f\ \omega^l\ s\cdot e$.
\end{pro} 
\begin{proof}
	The following are equivalent: $f\ \omega^l\ s\cdot e$; $f\in S(s\cdot e)$; $f\in Eq(s,se)$.
\end{proof}

It follows that in a modal left $E$-monoid $S$, $s\cdot e$ is the largest $f\in E$ (with respect to $\omega^l$) such that $fs=fse$.  

\begin{eg} A small non-inductive yet modal left $E$-monoid. \label{smallnoand} \end{eg}
\noindent Let $S=\{0,e,f,1\}$ be the band in which $0$ is a zero, $1$ is an identity, $ef=fe=0$, and $E=\{e,f,1\}$.  Then $S$ is easily verified to be a modal left $E$-monoid, and $e,f\ \omega^l\ 1$ but $e,f$ have no lower bound in $E$ under $\omega^l$. 
\medskip 

We next give a description of left $E$-modal operations which makes clear that they are indeed ``actions" in the usual sense. 

\begin{pro}  \label{modalaws}
	Suppose $S$ is a left $E$-monoid, and there is a function $\cdot: S\times E\rightarrow E$.  Then $S$ is modal with $\cdot$ the left $E$-modal operation if and only if $E$ is right pre-reduced and the following laws hold.  For all $e,f\in E$ and $s,t\in S$:
	\ben[label=\textup{(M\arabic*)}]
	\item $e\cdot e=1$, \label{m1}
	\item $s\cdot 1=1$, \label{m2}
	\item $(s\cdot e)s=(s\cdot e)se$, \label{m3}
	\item $(st)\cdot e=s\cdot(t\cdot e)$. \label{m4}
	\een
	Moreover, $S$ is an inductive left $E$-monoid if and only if the above laws hold and $E$ is a meet-semilattice with respect to $\omega^l$ such that the following law holds:
	\ben
	\item[\textup{(M5)}] $s\cdot (e\wedge f)=(s\cdot e)\wedge (s\cdot f)$, where $\wedge$ is meet in $(E,\omega^l)$. \label{m5}
	\een
\end{pro}
\begin{proof}
Suppose $S$ is modal.  Then $E$ is right pre-reduced, and the first three laws are immediate. For the fourth, letting $s,t\in S$ and $e\in E$, and letting $f=(st)\cdot e$ and $g=s\cdot(t\cdot e)$, we see that $fst=fste$, so $fs(t\cdot e)=fs$, and so $f\omega^l s\cdot(t\cdot e)=g$ by Proposition \ref{eqmod0cor}.  Conversely, $gs=gs(t\cdot e)$, so $gst=gs(t\cdot e)t=gs(t\cdot e)te=gste$, and so $g\ \omega^l\ (st)\cdot e=f$ using the same result.  Because $E$ is right pre-reduced, $f=g$.  So \ref{m4} is now established. 
	
Conversely, suppose \ref{m1}--\ref{m4} all hold and $E$ is right pre-reduced.  Suppose $ste=st$ for some $s,t\in S$ and $e\in E$.  Then using \ref{m1}, \ref{m2} and \ref{m4},
$$s\cdot(t\cdot e)=(st)\cdot e=(ste)\cdot e=(st)\cdot(e\cdot e)=(st)\cdot 1=1,$$ 
so by \ref{m3},
$$s(t\cdot e)=1s(t\cdot e)=(s\cdot(t\cdot e))s(t\cdot e)=(s\cdot(t\cdot e))s=s.$$
Conversely, if $s(t\cdot e)=s$ then by the third law, $ste=s(t\cdot e)te=s(t\cdot e)t=st$.  Hence $Eq(te,t)=S(t\cdot e)$ and so $S$ is modal.
	
Now suppose $S$ is an inductive left $E$-monoid.  Note that if $e,f\in E$ with $e\ \omega^l\ f$, then $e=ef$, so $(s\cdot e)sf=(s\cdot e)sef=(s\cdot e)se=(s\cdot e)s$, so $s\cdot e\ \omega^l\ s\cdot f$ by Proposition \ref{eqmod0cor}.  Hence because $e\wedge f\ \omega^l\ e,f$, $s\cdot (e\wedge f)\ \omega^l\ (s\cdot e)\wedge (s\cdot f)$.  Conversely, suppose $g\ \omega^l\  s\cdot e,g\ \omega^l\  s\cdot f$.  Then $gse=g(s\cdot e)se=g(s\cdot e)s=gs$ and similarly $gsf=gs$, so $gs(e\wedge f)=gs$ and so $g\ \omega^l\  s\cdot(e\wedge f)$.  Hence, (M5) holds.
	
Conversely, suppose $E$ is a meet semilattice with respect to $\omega^l$, and (M5) holds.  Then $\omega^l$ is a partial order and so $E$ is right pre-reduced.   Moreover, if $sf=s=sg$ for some $s\in S$ and $f,g\in E$, then $s\cdot f=(sf)\cdot f=s\cdot(f\cdot f)=s\cdot 1=1$ and similarly $s\cdot g=1$, so $s\cdot(e\wedge f)=(s\cdot e)\wedge (s\cdot f)=1\cdot 1=1$, and so $s(e\wedge f)=(s\cdot(e\wedge f))s(e\wedge f)=(s\cdot(e\wedge f))s=s$.  So $S$ is inductive.
\end{proof}

If $S$ is a modal left $E$-monoid, it follows easily that $1\cdot e=e$ for all $e\in E$, so all the conditions in the definition of an action of a monoid $S$ on a semilattice $E$ are satisfied when $S$ is an inductive left $E$-monoid.  

Next we show that a left $E$-modal monoid is also left $E'$-modal providing $E'\sim_l E$, and moreover if the former is inductive then so is the latter. 

\begin{pro}  \label{lmequiv}
Let $S$ be a left $E$-monoid, with $E'\sim_l E$, and suppose that for each $e\in E$, $e'\in E'$ is such that $e\, {\mc L}\, e'$.  Then $S$ is modal as a left $E'$-monoid, with left $E'$-modal operation $s\cdot e'=(s\cdot e)'$, for all $e\in E$.  If $S$ is inductive as a left $E$-monoid, then it is inductive as a left $E'$-monoid, and $(e\wedge f)'=e'\wedge f'$ for all $e,f\in E$.
\end{pro}
\begin{proof}
First, $E'$ is right pre-reduced.  Further, for $s,t\in S$ and $e'\in E'$, using the fact that $S$ is left $E$-modal as well as Lemma \ref{bighelp}, the following are seen to be equivalent: $tse'=ts$; $tse=ts$; $t(s\cdot e)=t$; $t(s\cdot e)'=t$; $t(s\cdot e')=t$.  So the left $E'$-monoid $S$ is also modal.

If $S$ is an inductive left $E$-monoid then $C_E(S)$ is inductive by Proposition \ref{isind}.  But by Proposition \ref{isomconst}, $C_E(S)\cong C_{E'}(S)$, so the latter is inductive, and so by Proposition \ref{converse}, $S$ is an inductive left $E'$-monoid.  Since $(E,\omega^l)$ and $(E',\omega^l)$ are isomorphic as posets, it follows that $(e\wedge f)'=e'\wedge f'$ for all $e,f\in E$. 
\end{proof}

Even in a modal left $E$-monoid, where meets of elements of $E$ need not exist, we obtain the following.

\begin{lem}  \label{ef}
	Let $S$ be a modal left $E$-monoid, with $e\in E(S)$ and $f\in E$.  Then $(e\cdot f)e$ is a greatest lower bound of $e,f$ in $E(S)$ with respect to $\omega^l$, and if $se=sf=s$ for some $s\in S$, then $s(e\cdot f)e=s$.
\end{lem}
\begin{proof}
	Choose $e\in E(S)$, $f\in E$.  Then using Proposition \ref{modalaws}, $e\cdot f=(ee)\cdot f=e\cdot(e\cdot f)$, so 
	$$(e\cdot f)e=(e\cdot(e\cdot f))ee=(e\cdot(e\cdot f))e(e\cdot f)e=(e\cdot f)e(e\cdot f)e,$$
	and so $(e\cdot f)e\in E(S)$. 
	If $s\in S$ is such that $se=sf=s$, then $sef=se$, so $s(e\cdot f)=s$, and so $s(e\cdot f)e=se=s$.
	
	Finally, $(e\cdot f)e=(e\cdot f)ef$ and so $(e\cdot f)e\omega^l e,f$ in $E(S)$. If $g\omega^l e,f$ for some $g\in E(S)$, then $ge=g=gf$, so letting $s=g$ in what was just shown gives $g(e\cdot f)e=g$, and so $g\omega^l (e\cdot f)e$.
\end{proof}

This last result ensures that the meet of elements of $E$ in an inductive left $E$-monoid admits a description ``modulo ${\mc L}$".  But more can be said.  

\begin{pro}  \label{conjef}
	Let $S$ be a modal left $E$-monoid.  Then it is inductive if and only if for all $e,f\in E$, $(e\cdot f)e\, {\mc L}\, g$ for some $g\in E$, in which case $e\wedge f=g$. In this case, $e\wedge f=e\wedge (e\cdot f)$.  
	If $S$ is inductive and $E$ is right reduced, then $e\wedge f=(e\cdot f)e$.
\end{pro}
\begin{proof}
	Suppose that for all $e,f\in E$, the idempotent $(e\cdot f)e\, {\mc L}\, g$ for some $g\in E$.  Now $(e\cdot f)e$ is a greatest lower bound within $E(S)$ of $e,f\in E$ by Lemma \ref{ef}.  Hence $g$ must be their greatest lower bound in $E$.  If $se=s=sf$ for some $s\in S$, then by Lemma \ref{ef}, $s(e\cdot f)e=s$, and so by the first part of Lemma \ref{bighelp}, $sg=s$ as well.  Hence $S$ is inductive.     
	
	Conversely, suppose $S$ is inductive.  By Lemma \ref{ef}, $(e\cdot f)e$ is a greatest lower bound of $e,f$ in $E(S)$, so $e\wedge f\ \omega^l\  (e\cdot f)e$.  Conversely, $(e\cdot f)ee=(e\cdot f)e$ and $(e\cdot f)ef=(e\cdot f)e$, so $((e\cdot f)e)(e\wedge f)=(e\cdot f)e$, so $(e\cdot f)e\ \omega^l\  e\wedge f$.  So $e\wedge f \, {\mc L}\, (e\cdot f)e$.  
	
	Again, if $S$ is inductive, then, for $e,f\in E$, working in $Rest(E,S)$ we have $(e\wedge f,e\wedge f)=(e,e)(f,f)=(e\wedge (e\cdot f),e\wedge (e\cdot f))$, so $e\wedge f=e\wedge (e\cdot f)$.
	
	Suppose $E$ is right reduced.  Then $\omega^l$ is the natural order $\omega$ on $E$.  Hence for all $e,f\in E$, $(e\wedge f)ef=(e\wedge f)e$, so $e\wedge f\ \omega\ e\cdot f$, and so because $e\wedge f\ \omega\ e$ and $(e\cdot f)e\, {\mc L}\, e\wedge f$, we have 
	$$(e\cdot f)e=(e\cdot f)e(e\wedge f)=(e\cdot f)(e\wedge f)=e\wedge f,$$
	as claimed.
\end{proof}

\begin{dfn}
The modal left $E$-monoid has {\em definable meets} if $(e\cdot f)e\in E$ for all $e,f\in E$. 
\end{dfn}

From the previous result, a modal left $E$-monoid with definable meets is inductive, with meet in $E$ given by $e\wedge f=(e\cdot f)e$ for all $e,f\in E$; moreover any modal left $E$-monoid in which $E$ is right reduced has definable meets.  

Another way to think about modal left $E$-monoids is in terms of arbitrary sets of idempotents.

\begin{dfn}
Let $S$ be a monoid with $1\in E\subseteq E(S)$.  We say $S$ is {\em (left) $E$-protomodal} if for every $e\in E$, and $s\in S$, $Eq(s,se)$ is non-empty and generated as a left ideal by a member of $E$.  If $E=E(S)$, we say $S$ is {\em protomodal}.
\end{dfn}

So if $E$ is right pre-reduced, then ``left $E$-protomodal = left $E$-modal".

\begin{dfn}
Let $S$ be a monoid with $1\in E\subseteq E'\subseteq E(S)$.  We say {\em $E$ is maximal right pre-reduced in $E'$} if $E$ is right pre-reduced and every element of $E'$ is related by ${\mc L}$ to a (necessarily unique) member of $E$.
\end{dfn}

If $E'=E(S)$ in the above, then we recover the definition of $E$ being maximal right pre-reduced.
Evidently if $F,G$ are maximal right pre-reduced in $E\subseteq E(S)$, then $F\sim_l G$.

Not surprisingly, $E$-protomodal monoids are closely related to left $E$-modal monoids.

\begin{pro}  \label{Eprot}
Let $S$ be a monoid with $1\in F\subseteq E(S)$, and suppose $E$ is right pre-reduced in $F$.  Then $1\in E$, and $S$ is $F$-protomodal if and only if $S$ is a modal left $E$-monoid.  If also $S$ is integral and $0\in F$, then $0\in E$ and $S$ is $F$-protomodal if and only if $S$ is a modal left $E$-monoid with zero.
\end{pro}
\begin{proof}
Let $S,F$ and $E$ be as in the proposition statement.

That $1\in E$ follows because $1$ is the only idempotent in its ${\mc L}$-class. If $S$ is integral and $0\in F$ then $(f,0)\in {\mc L}$ for some $f\in E$, so $f=f0=0$, and so $0\in E$.

Suppose $S$ is $F$-protomodal.  Pick $e\in E\subseteq F$.  Then for $s\in S$, $Eq(s,se)$ is generated by some $f\in F$, hence by some $f'\in E$ with $f\, {\mc L}\, f'$ by Proposition \ref{eq}.  Because $E$ is right pre-reduced, $S$ is a modal left $E$-monoid.  If $S$ is integral and $0\in F$, then as above $0\in E$ and so $S$ is a left $E$-monoid with zero, hence a modal one.

Conversely, suppose $S$ is a modal left $E$-monoid (with or without zero).  For $s\in S$ and $e'\in F$, there is $e\in E$ such that $e\, {\mc L}\, e'$, and then there is $f\in E\subseteq F$ that generates $Eq(s,se)=Eq(s,se')$ (using Proposition \ref{eq}), so $S$ is $F$-protomodal. 
\end{proof}

\begin{cor} \label{eqmod}
Let $S$ be a monoid.  The following are equivalent:
\ben
\item $S$ is a modal left $E$-monoid for some maximal right pre-reduced $E\subseteq E(S)$;
\item $S$ is a modal left $E$-monoid for every maximal right pre-reduced $E\subseteq E(S)$;
\item $S$ is protomodal.  
\een
\end{cor}

There are monoids which are not protomodal.  Let $S=\{0,1,a\}$ with $0$ a zero, $1$ an identity and $a^2=0$; then $Eq(a,a0)=Eq(a,0)=\{0,a\}$, which is not generated by an idempotent, so $S$ is not protomodal. 

Even if a monoid is not protomodal, it will always have at least a rather trivial inductive left $E$-monoid structure for some choice of $E$.  

\begin{pro}  \label{smallind}
Every monoid $S$ is an inductive left $\{1\}$-monoid, with $t\cdot 1=1$ for all $t\in S$.  Every integral monoid with zero is an inductive left $\{0,1\}$-monoid, with $t\cdot 1=1$ for all $t\in S$, and $t\cdot 0=0$ for non-zero $t$, with $0\cdot 0=1$.  
\end{pro}
\begin{proof}
The first assertion is almost trivial. 

For the second, clearly $Eq(t1,t)=S1$.  For non-zero $t$, $st0=st$ if and only if $s=0=s0$, so $Eq(t0,t)=S0$, and $Eq(0,0^2)=S=S1$.  So $S$ is $\{0,1\}$-protomodal, and hence is a modal $\{0,1\}$-monoid since $\{0,1\}$ is obviously right pre-reduced.  The condition \ref{i2'} is easily checked by a case analysis.
\end{proof}

More interestingly, we next show that every monoid has a largest $E$ (modulo ${\mc L}$) such that it is modal as a left $E$-monoid, and in that case it is inductive as a left $E$-monoid.  (Indeed in the above three-element example $S$, this largest choice is obviously $E=\{1\}$.)  

\begin{pro}  \label{maxmodE}
Let $S$ be a monoid, and define 
\[
G=\bigcup \{E\subseteq E(S)\mid S\mbox{ is $E$-protomodal}\}.
\] 
Then $S$ is $G$-protomodal, and if $S$ is also $H$-protomodal, then $H\subseteq G$.  Choosing $E'$ to be maximal right pre-reduced in $G$, $S$ is a modal left $E'$-monoid, and $E'$ is largest possible in the sense that if $S$ is a modal left $E$-monoid, then $E$ is right-equivalent to a subset of $E'$.  Moreover the left $E'$-monoid $S$ is inductive.
\end{pro}
\begin{proof}
First note that $G\neq\emptyset$ by Proposition \ref{smallind}.
Suppose $s\in S$ and $e\in G$.  Then $e\in E$ for some $E\subseteq E(S)$ for which $S$ is $E$-protomodal, so there is $f\in E\subseteq G$ that generates $Eq(s,se)$.  So $S$ is $G$-protomodal, and so by Proposition \ref{Eprot}, $S$ is a modal left $E'$-monoid. Otherwise, only the final statement does not follow immediately.  

Define $F=\{(e\cdot f)e\mid e,f \in E'\}$.  Then $E'\subseteq F$ since for all $e\in E$, $e=(e\cdot e)e\in F$. Let $s\in S$ and $e\in F$.  Then $e=(f\cdot g)f$ for some $f,g\in E'$.  We show that $Eq(s,se)$ is generated by $h=((s\cdot f)\cdot (s\cdot g))(s\cdot f)\in F$, as computed in the modal left $E'$-monoid $S$.  Because $(s\cdot f)sf=(s\cdot f)s$, $(s\cdot g)sg=(s\cdot g)s$ and $h(s\cdot f)=h(s\cdot f)(s\cdot g)$, we have that $hsfg=hsg=hs=hsf$.  Hence, $hs\in Eq(fg,g)$, and so $hse=hs(f\cdot g)f=hsf=hs$, giving that $h\in Eq(se,s)$.  Suppose $t\in Eq(se,s)=Eq(s(f\cdot g)f,s)$.  Then $ts=ts(f\cdot g)f=ts(f\cdot g)fg$, so $ts=tsf=tsg$, so $t(s\cdot f)=t(s\cdot g)=t$, and so by Lemma \ref{ef}, $th=t$.   So $Eq(se,s)=Sh$. Hence, $S$ is $F$-protomodal and so $F\subseteq G$.  It follows from the definition of $E'$ that for all $e,f\in E'$, there is $g\in E'$ such that $g\, {\mc L}\, (e\cdot f)e\in F\subseteq G$, so by Proposition \ref{conjef}, the modal left $E'$-monoid $S$ is inductive. 
\end{proof}

\begin{cor}
Every modal left $E$-monoid $S$ is an inductive left $E'$-monoid for some $E'$ containing $E$. 
\end{cor}

\begin{cor}  \label{cormax}
If $S$ is a modal left $E$-monoid in which $E$ is maximal right pre-reduced, then it is inductive, and for $e,f\in E$, $e\wedge f\, {\mc L}\, (e\cdot f)e$.
\end{cor}

\begin{cor}  
If $S$ is protomodal, then for any maximal right pre-reduced $E\subseteq E(S)$, $S$ is an inductive left $E$-monoid such that $e\wedge f\, {\mc L}\, (e\cdot f)e$ for all $e,f\in E$.
\end{cor}

Recall from Definition \ref{restdef} that if $S$ is an inductive left $E$-monoid (as in Definition \ref{inddef}), then
$$Rest(E,S)=\{(e,s)\mid e\in E, s\in S,es=s\},$$
a left restriction monoid with multiplication given by $$(e,s)(f,t)=(e\wedge (s\cdot f), (e\wedge (s\cdot f))st),$$ 
and in which $D((e,s))=(e,e)$.  If $S$ is integral, $Rest(E,S)$ contains as a submonoid closed under $D$ the left restriction monoid with zero
$$Rest_0(E,S)=\{(e,s)\in Rest(E,S) \mid e=0\Rightarrow s=0\}.$$

\begin{pro}  \label{restS}
Suppose $S$ is protomodal.  Then for $E,E'$ maximal right pre-reduced sets of idempotents, $Rest(E,S)\cong Rest(E',S)$ as left restriction monoids, and if $S$ is integral then $Rest_0(E,S)\cong Rest_0(E',S)$.
\end{pro}
\begin{proof}
Now $S$ is an inductive left $E$-monoid and an inductive left $E'$-monoid, integral if $S$ is integral as a monoid.  Of course $E\sim_l E'$, so by Proposition \ref{0agrees}, $Rest(E,S)\cong Rest(E',S)$, and if $S$ is integral, $Rest_0(E,S)\cong Rest_0(E',S)$.
\end{proof}

\begin{dfn}  \label{restSdef}
Suppose $S$ is protomodal.  If $E$ is a maximal right pre-reduced set of idempotents of $S$, denote the isomorphism class of $Rest(E,S)$ (and $Rest_0(E,S)$ if $S$ is integral) by $Rest(S)$ (respectively $Rest_0(S)$).
\end{dfn}

Informally, we will write ``$A\cong Rest(S)$" rather than ``$A\in Rest(S)$".  In fact we can also obtain $Rest(S)$ as a left restriction semigroup defined independently of the choice of $E$, as follows.  (We omit the details, which are routine to check.)

Suppose $S$ is protomodal, and form the constellation $Q=C_{E(S)}(S)$.  Define the equivalence relation $\theta$ on $Q$ as follows: set $(e,s)\mathrel{\theta} (f,t)$ if and only if $e\, {\mc L}\, f$ and $es=et$ (equivalently, $fs=ft$).  This turns out to be a congruence on $Q$ which when factored out gives a constellation $Q/\theta\cong C_E(S)$ where $E$ is any maximal right pre-reduced set of idempotents of $S$, and so the induced left restriction semigroup on the inductive constellation $Q/\theta$ is in $Rest(S)$.  Something very similar can be done for $Rest_0(S)$.  This approach highlights the irrelevance of the choice of any particular maximal right pre-reduced set of idempotents $E\subseteq E(S)$, since all of $E(S)$ is used in the construction.  

\begin{eg} A small example.  \label{smalleg} \end{eg}

Consider the magma $S=\{1,e,f,g,s\}$, given by the following multiplication table. 
$$
\begin{array}{c|ccccc}
\times&e&1&f&g&s\\
\hline
e&e&e&e&e&e\\
1&e&1&f&g&s\\
f&e&f&f&e&e\\
g&s&g&s&g&s\\
s&s&s&s&s&s
\end{array}.
$$

It is routine to check that $T=\{e,f,g,s\}$ is a subalgebra.  Among its congruences are $\theta$ defined by the partition $\{e,f\},\{s,g\}$, and $\delta$ defined by the partition $\{e,s\},\{f\},\{g\}$.  It is easily seen that $T/\theta$ is a two-element left-zero semigroup, whereas $T/\delta$ is a semilattice (in which $\{e,s\}$ is the meet of $\{f\}$ and $\{g\}$).  Moreover $\theta\cap\delta$ is the diagonal relation, so $T$ is a subdirect product of these two semigroups and hence is itself a semigroup, which is evidently a band.  Indeed $T$ is a semilattice of left zero semigroups (namely its $\delta$-classes).  Clearly, $S$ is the semigroup $T$ with adjoined identity element and hence is a monoid.

Let $E=\{1,e,f,g\}$, a maximal right pre-reduced subset of $E(S)=S$, since $(e,s)\in {\mc L}$ but otherwise no two elements of $S$ are ${\mc L}$-related.  Clearly $e\ \omega^l\  f,g\ \omega^l\  1$.  It is not hard to check that there is a left $E$-modal operation defined as follows:

$$
\begin{array}{c|cccc}
\cdot&e&1&f&g\\
\hline
e&1&1&1&1\\
1&e&1&f&g\\
f&g&1&1&g\\
g&f&1&f&1\\
s&1&1&1&1
\end{array}.
$$

Since $E$ is maximal right pre-reduced, the modal left $E$-monoid $S$ is inductive by Corollary \ref{cormax}.  
Note that $f\wedge g=e$ but $(f\cdot g)f=gf=s$.  The only other maximal right pre-reduced choice involves replacing $e$ by $s$ in $E$ to give $E'=\{1,s,f,g\}$; it follows from Proposition \ref{lmequiv} that $S$ is an inductive left $E'$-monoid also.  
So letting $e'=s,s'=e$ with $f'=f,g'=g$ and $1'=1$, we see that as computed in $S$, $(g'\cdot f')g'=(g\cdot f)'g'=f'g'=e=g$, while $g'\wedge f'=(g\wedge f)'=e'=s$ so they are still unequal.  This shows that there are protomodal left $E$-monoids having no choice of maximal right pre-reduced $E$ for which the associated inductive $E$-monoid has definable meets. 

Note that $S'=\{e,1,f,s\}$ is a submonoid of $S$ above.  Letting $E''=\{1,f,s\}$, it is easy to see that $S'$ is an inductive left $E''$-monoid, and indeed that $h\wedge k=(h\cdot k)h$ for all $h,k\in E''$, yet $E''$ is not right reduced.  So a modal left $E$-monoid with definable meets need not have $E$ right reduced, showing that the converse of the final part of Proposition \ref{conjef} does not hold in general.

\section{Connection with the Zappa-Sz\'{e}p product}  \label{Zappa}

At this point it is worth noting that one could use an inductive $E$-monoid $S$ to define a semidirect product in the usual way, since $S$ acts on the semilattice $E$ by \ref{m2}, \ref{m4} and (M5).  Thus, we could define, for all $(e,s),(f,t)\in E\times S$,
$$(e,s)*(f,t)=(e\wedge (s\cdot f),st),\ D((e,s))=(e,1),$$ 
and the result is well-known to be a proper left restriction semigroup which we call $P(E,S)$ here; see \cite{manes} for example.  This is different to our definition of both domain and multiplication in $Rest(E,S)$, and the underlying set is all of $E\times S$.  However, we do have the following.

\begin{pro}  \label{semihom}
Let $S$ be an inductive left $E$-monoid.  Then the mapping $\psi: P(E,S)\rightarrow Rest(E,S)$ given by $(e,s)\psi=(e,es)$ is a surjective left restriction semigroup homomorphism which separates domain elements. 
\end{pro}
\begin{proof}
First, note that $\psi$ is surjective, since it fixes all elements of $C_E(S)\subseteq P(E,S)$.  If $(e,s)\in P(E,S)$, then $D((e,s))\psi=(e,1)\psi=(e,e)=D((e,es))=D((e,s)\psi)$, so $\psi$ respects $D$.  For $(e,s),(f,t)\in P(E,S)$, first note that for any $g\in E$, the following are equivalent: $g\ \omega^l\  e\wedge ((es)\cdot f)$; $g\ \omega^l\  e$ and $g\ \omega^l\  (es)\cdot f$; $g=ge$ and $ges=gesf$; $g=ge$ and $gs=gsf$; $g\ \omega^l\  e$ and $g\ \omega^l\  s\cdot f$; $g\ \omega^l\  e\wedge (s\cdot f)$.  (Here we have twice used \ref{i1'} in Corollary \ref{Emodcor}.)  Hence $e\wedge ((es)\cdot f)=e\wedge (s\cdot f)$.  Thus,
\bea
((e,s)\psi)((f,t)\psi)&=&(e,es)(f,ft)\\
&=&(e\wedge ((es)\cdot f),(e\wedge ((es)\cdot f))esft)\\
&=&(e\wedge (s\cdot f),(e\wedge (s\cdot f))esft)\\
&=&(e\wedge (s\cdot f),(e\wedge (s\cdot f))sft)\\
&=&(e\wedge (s\cdot f),(e\wedge (s\cdot f))(s\cdot f)sft)\\
&=&(e\wedge (s\cdot f),(e\wedge (s\cdot f))(s\cdot f)st)\\
&=&(e\wedge (s\cdot f),(e\wedge (s\cdot f))st)\\
&=&((e\wedge (s\cdot f),st)\psi\\
&=&((e,s)*(f,t))\psi,
\eea
so $\psi$ respects multiplication.  Trivially $\psi$ separates domain elements.
\end{proof}  

So every left restriction monoid with enough large idempotents is a domain-separating homomorphic image of a proper left restriction semigroup, specifically of a semidirect product of a monoid with a semilattice.  Theorem 4.6 in \cite{manes} shows that every left restriction semigroup is a projection-separating homomorphic image of a proper one which in this case is only a subsemigroup of a semidirect product of a monoid with a semilattice.

In \cite{szendrei}, Szendrei considers a two-sided restriction semigroup $S$ (meaning that $S$ is both a left and right restriction semigroup with the property that $D(S)=R(S)$).  Such an $S$ is said to be {\em left factorizable} if $S=D(S)S_1=\{em\mid e\in D(S),m\in S_1\}$, where $S_1=\{s\in S\mid D(s)=1\}$.  Szendrei shows that if $S$ is a restriction monoid, then it is left factorizable if and only if it is a projection separating homomorphic image (image of a semigroup homomorphism respecting $D$ and $R$) of a W-product of a semilattice by a monoid (defined as a kind of two-sided semidirect product); see Proposition 3.4 and Theorem 3.6 in \cite{szendrei}.  Note that the left restriction monoid $Rest(E,S)$ is also ``left factorizable", as Corollary \ref{factor} shows, and Proposition \ref{semihom} shows it is a projection-separating homomorphic image of a semidirect product of a monoid with a semilattice.  Perhaps there is a characterisation of left factorizability of left restriction monoids along these lines.

Returning to left restricion monoids with enough large idempotents, there is a nice connection with Zappa-Sz\'{e}p products as in \cite{kunze}; these generalise semidirect products.  For groups, monoids and indeed semigroups, we say $S$ is an internal Zappa-Sz\'{e}p product of subsemigroups $M_1,M_2$ if every $s\in S$ may be uniquely expressed as $s=m_1m_2$ where $m_1\in M_1, m_2\in M_2$.  If $S$ is a group/monoid, we require $M_1,M_2$ to be subgroups/submonoids, and in these cases there is an external characterisation of Zappa-Sz\'{e}p products in terms of actions. 

Very generally, one may define an external Zappa-Sz\'{e}p product of semigroups if each acts on the other in a certain way. We here follow the notational conventions used in \cite{zenab}, modified to have some notational similarity with the current situation.  Thus suppose $S,E$ are semigroups with multiplication in $S$ denoted by juxtaposition and in $E$ by $\wedge$, and that we have actions $S\times E\rightarrow E$ given by $(s,e)\mapsto s\cdot e$, and $S\times E\rightarrow S$ given by $(s,e)\mapsto s^e$, satisfying the following four conditions: for all $s,t\in S$ and $e,f\in E$,

\ben[label=\textup{(ZS\arabic*)}]
\item $(st)\cdot e=s\cdot (t\cdot e)$,  \label{zs1}
\item $s\cdot (e\wedge f)=(s\cdot e)\wedge (s^e\cdot f)$,  \label{zs2}
\item $(s^e)^f=s^{e\wedge f}$, \label{zs3}
\item $(st)^e=s^{t\cdot e}t^e$. \label{zs4}
\een

If these conditions hold, then define, for all $(e,s),(f,t)\in E\times S$, 
$$(e,s)\otimes (f,t)=(e\wedge(s\cdot f), (s^f)t)).$$  
Then $\otimes$ is an associative operation, and $E\bowtie S=(E\times S,\otimes)$ is the external Zappa-Sz\'{e}p product of $E$ and $S$. There is an obvious similarity with the definition of multiplication in $Rest(E,S)$, but note that $(e,s)\in Rest(E,S)$ only if $es=s$.

For the external Zappa-Sz\'{e}p product to correspond to the internal version, at least for monoids, one assumes four further conditions involving the identity elements, but there is still interest in external Zappa-Sz\'{e}p products of semigroups satisfying only \ref{zs1} to \ref{zs4}. 

\begin{pro}  \label{ZSP}
Let $S$ be an inductive left $E$-monoid.  Define $s\cdot e$ using the left $E$-modal operation and $s^e=(s\cdot e)s$.  Then Laws \ref{zs1}, \ref{zs2} and \ref{zs4} hold. 
Moreover \ref{zs3} holds if and only if $E$ has definable meets, and in that case, $Rest(E,S)$ is a subsemigroup of the Zappa-Sz\'{e}p product $E\bowtie S$.  

In this case, ${\hat E}=\{(e,e)\mid e\in E\}$ is a semilattice isomorphic to $(E,\wedge)$, and $Rest(E,S)$ is the maximum left restriction subsemigroup of $E\bowtie S$ with set of domain elements ${\hat E}$.  If $S$ is with zero, then $(0,0)$ is a zero element of $E\bowtie S$ and $Rest_0(E,S)$ is the maximum left restriction subsemigroup with zero of $E\bowtie S$ with set of domain elements ${\hat E}$.
\end{pro}
\begin{proof}  \ref{zs1} is immediate.  For \ref{zs2}, $$(s\cdot e)\wedge (s^e\cdot f)=(s\cdot e)\wedge (((s\cdot e)s)\cdot f)=(s\cdot e)\wedge ((s\cdot e)\cdot (s\cdot f))=(s\cdot e)\wedge (s\cdot f)=s\cdot(e\wedge f)$$
using Proposition \ref{conjef}.  For \ref{zs4}, $s^{t\cdot e}t^e=(s\cdot(t\cdot e))s(t\cdot e)t=(s\cdot (t\cdot e))st=((st)\cdot e)st=(st)^e$.

Now suppose $E$ has definable meets.  Then for $s\in S$ and $e,f\in E$, 
$$(s^e)^f=(((s\cdot e)s)\cdot f)((s\cdot e)s)=((s\cdot e)\cdot (s\cdot f))(s\cdot e)s=((s\cdot e)\wedge (s\cdot f))s=(s\cdot (e\wedge f))s=s^{e\wedge f},$$ establishing \ref{zs3}.  Conversely, if \ref{zs3} holds, note that for all $e\in E$, $1^e=(1\cdot e)1=e$, and so putting $s=1$ in \ref{zs3} gives $(1^e)^f=1^{e\wedge f}$ for all $e,f\in E$, which is to say that $e^f=e\wedge f$, so $(e\cdot f)e=e\wedge f$, so $E$ has definable meets.

So if $E$ has definable meets, then for $(e,s),(f,t)\in E\bowtie S$ for which $es=s$ and $ft=t$, we have $(e\wedge (s\cdot f))s=(e\cdot(s\cdot f))es=((es)\cdot f)es=(s\cdot f)s$, and so
$$(e,s)\otimes (f,t)=(e\wedge s\cdot f, (s\cdot f)st)=(e\wedge (s\cdot f),(e\wedge(s\cdot f)st)=(e,s)(f,t)$$ as in $Rest(E,S)$.

Assuming definable meets, the fact that $\hat E$ is as described is easily seen, as is the fact that for $(e,s)\in E\bowtie S$, if $(f,f)\otimes(e,s)=(e,s)$ for some $f\in E$, then 
$$(e,s)=(f\wedge (f\cdot e),(f\cdot e)fs)=(f\wedge e,(f\wedge e)s),$$ 
giving that $f\wedge e=e$ and $(f\wedge e)s=s$, so $e\ \omega^l\  f$ and $es=s$, implying that $(e,s)\in Rest(E,S)$.  It follows that any left restriction semigroup in $E\bowtie S$ with domain elements ${\hat E}$ is a subset of $Rest(E,S)$.  If $S$ is an inductive left $E$-monoid with zero, then for $(e,0)\in E\bowtie S$, we have 
$$(e,0)(0,0)=(e\wedge (0\cdot 0),(0\cdot 0)0)=(e,0),$$ 
so any left restriction subsemigroup of $E\bowtie S$ which has zero $(0,0)$ cannot include elements $(e,0)$ if $e\neq 0$.  So $Rest_0(E,S)$ is the largest such.
\end{proof}  

Starting with a left restriction semigroup $S$, a similar construction to the above is used in Section 3 of \cite{zenab}, by letting $E=D(S)$, and by setting $s\cdot e=D(se)$ and $s^e=se=D(se)s=(s\cdot e)s$.  It is shown that \ref{zs1}--\ref{zs4} are satisfied in this case, and so $E\bowtie S$ exists, and the subset $T=\{(e,s)\in E\times S\mid es=s\}$ is shown to be the largest left restriction subsemigroup of $E\bowtie S$ in which $D(T)=\{(e,e)\mid e\in D(S)\}$ (see Theorem $3.4$ of \cite{zenab}).  However, this is not a special case of our result since $S$ equipped with the above action is not a modal left $E$-monoid even if $S$ has identity, since $s\cdot 1=D(s)$ rather than $1$.  This is manifest in the structure of $T$ since the induced constellation product $(e,s)\circ (f,t)$ exists if and only if $e=D(s)$ and $sf=s$, a stronger condition than ours.

\section{Applications}  \label{egs}

Courtesy of the results of Section \ref{lrchar}, it is quite straightforward to identify and then describe specific left restriction monoids as semigroup left $E$-completions.  In this section, we do this for three important classes of examples.

First, recall once again (from Definition \ref{restdef}) that if $S$ is an inductive left $E$-monoid, then 
$$Rest(E,S)=\{(e,s)\mid e\in E, s\in S,es=s\},$$
a left restriction semigroup with multiplication given by $$(e,s)(f,t)=(e\wedge (s\cdot f), (e\wedge (s\cdot f))st),$$ 
and in which $D((e,s))=(e,e)$, and that if $S$ is integral it has left restriction subsemigroup
$$Rest_0(E,S)=\{(e,s)\in Rest(E,S) \mid e=0\Rightarrow s=0\}.$$  Recall also that for a semigroup $S$, $S^0$ denotes $S$ with adjoined zero.

\subsection{Transformations and partial functions.}  \label{pareg}

The canonical examples of semigroups and left restriction semigroups (in the sense that all embed in such examples) are $T_X$ and $PT_X$ respectively.  We begin with a result relating them.  

First, we note that the idempotent elements of $T_X$ are precisely the projections onto non-empty subsets of $X$: mappings that fix their range.  It is easy to see that for $e,f\in E(T_X)$, $e\ \omega^l\  f$ if and only if $\ran(e)\subseteq \ran(f)$ (where $\ran(e)$ denotes the range of $e$), and so $e\, {\mc L}\, f$ if and only if $\ran(e)=\ran(f)$.  Hence, a maximal right pre-reduced subset of $T_X$ consists of precisely one projection onto every possible non-empty subset of $X$.

If $S$ is protomodal, recall the definitions of $Rest(S)$ (and $Rest_0(S)$ where it makes sense) as in Definition \ref{restSdef}: it is the isomorphism class of $Rest(E,S)$ (and $Rest_0(E,S)$ if $S$ is integral) as in Definition \ref{restdef}, where $E\subseteq E(S)$ is maximal right pre-reduced; this is well-defined (independent of the choice of $E$) by Proposition \ref{restS}.

\begin{thm}  \label{TX0extend}
$T_X^0$ is a protomodal monoid, and $PT_X\cong Rest_0(T_X^0)$.
\end{thm}
\begin{proof}
Since $D(0)=0$, the empty function, $PT_X$ is a left restriction monoid with zero.  We identify $S=(PT_X)_1$ with $T_X$ as a monoid.  For each non-zero $e\in D(PT_X)$, pick precisely one $e'\in E(T_X)$ for which $\ran(e')=\ran(e)=\dom(e)$, and also set $0'=0$.  It follows easily that $e'e=e'$ and $ee'=e$, since $e'$ fixes its range, and so $e\, {\mc L}\, e'$.   Hence $F=\{e'\mid e\in D(PT_X), e\neq 0\}$ is a set of enough large idempotents of $PT_X$.  It follows from Theorem \ref{char0} that $T_X^0$ is an integral inductive left $E$-monoid in which $E=F\cup\{0\}$ is right pre-reduced, and the left $E$-modal operation is $s\cdot e'=D(se)'$ if $s\neq 0$, with $0\cdot e'=1$, for all $e'\in E$, and $PT_X\cong Rest_0(E,T_X^0)$.  But choosing any $f\in E(T_X)$, pick $e\in D(PT_X)$ such that $\ran(f)=\ran(e)=\dom(e)$; again, $f\, {\mc L}\, e$, and so $PT_X$ has precisely enough large idempotents.  It follows from Proposition \ref{max} that $E$ is maximal right pre-reduced, so $T_X^0$ is protomodal by Corollary \ref{eqmod}, and so $PT_X\cong Rest_0(T_X^0)$.
\end{proof}

Hence the entire structure of the (left restriction) monoid $PT_X$ may be recovered from that of the protomodal monoid $T_X^0$ via zero-reduced left $E$-completion, where $E$ is any maximal right pre-reduced set of idempotents in $E(T_X^0)$.  By contrast, consider the symmetric inverse monoid $I_X$ on the set $X$, which is a sub-left restriction monoid of $PT_X$ in which $D(s)=ss'$ for all $s\in I_X$.  If $X$ is finite, then $(I_X)_1$ is the symmetric group on $X$, and so certainly does not have enough large idempotents and so $I_X$ cannot be a left $E$-completion, zero-reduced or not.  

Because every left restriction semigroup embeds in $PT_X$ for some $X$, we immediately obtain the following.

\begin{cor}  \label{0embed}
Every left restriction semigroup embeds in $Rest(S)$ for some protomodal monoid $S$.
\end{cor}

There are choices of maximal right pre-reduced $E$ in $E(T_X^0)$ that are right reduced, and this allows us to relate our results to Zappa-Sz\'{e}p products of semigroups using Proposition \ref{ZSP}.  To show this for arbitrary $X$, we need one form of the axiom of choice (namely, the well-ordering theorem). 

Let $X$ be a set, and equip it with a well-order $\leq$.  Then for $e\in E(T_X)$, denote by $u_e\in X$ the smallest element of its range under $\leq$.

\begin{lem}  \label{rrE}
Define $E\subseteq E(T_X)$ as follows:
$$E=\{e\in E(T_X)\mid e(x)\neq x\mbox{ implies } e(x)=min(\ran(e))\}.$$
Then $E$ is maximal right pre-reduced and indeed is right reduced.
\end{lem}
\begin{proof}
Define $E$ as above.  Evidently there is precisely one choice of $e\in E$ for every possible range, so $E$ is maximal right pre-reduced. Suppose $e,f\in E$ and $ef=e$, so $\ran(e)\subseteq \ran(f)$.  Now for $x\in \ran(f)$, $xfe=xe$, while for $x\not\in \ran(f)$, $xfe=u_fe=u_e=xe$ since either $u_f=u_e$ or else $u_f\not\in \ran(e)$.  So $fe=e$.  Hence $E$ is right reduced.
\end{proof}

It is now straightforward to add $0$ to such $E$ to get a right reduced maximal right pre-reduced subset of $E(T_X^0)$.  From Lemma \ref{rrE}, Theorem \ref{TX0extend} and Proposition \ref{ZSP}, we obtain the following.

\begin{cor}
Let $E$ be a right reduced maximal right pre-reduced subset of $E(T_X^0)$.  Then $PT_X$ embeds as the largest left restriction semigroup with zero within the Zappa-Sz\'{e}p product $E\bowtie S$ having domains elements ${\hat E}$ as in Proposition \ref{ZSP}.
\end{cor}

\begin{eg} Constructing $PT_X$ from $T_X^0$ when $|X|=2$.   \label{2con}
\end{eg}
To give an illustration of the construction in the simplest non-trivial case, let $X=\{x,y\}$ be a two-element set, and recall the enumeration of the elements of $S=PT_X$ given in Example \ref{enougheg}.  Then $T_X$ is isomorphic to $S_1$, and $T_X^0$ is isomorphic to $S_1\cup\{0\}=\{1,0,p_x,p_y,i\}$.  There is only one possible choice of $E$ in this case, namely $E=E(S)=\{0,1,p_x,p_y\}$, which is therefore right reduced.  The elements of $Rest_0(E,S)$ are then easily seen to be
\[
(0,0), (p_x,p_x), (p_y,p_y), (p_x,p_y), (p_y,p_x), (1,1), (1,i), (1,p_x), (1,p_y), 
\]
corresponding to partial functions as follows
\[
(0,0)\leftrightarrow \emptyset,\ (p_x,p_x)\leftrightarrow e_x,\ (p_y,p_y)\leftrightarrow e_y,\ (p_x,p_y)\leftrightarrow j,
\]
\[ 
(p_y,p_x)\leftrightarrow k,\ (1,1)\leftrightarrow 1,\ (1,i)\leftrightarrow i,\ (1,p_x)\leftrightarrow p_x,\ (1,p_y)\leftrightarrow p_y.
\]

For example, since $D((p_x,p_y))=(p_x,p_x)$, and $(p_x,p_y)=(p_x,p_x)(1,p_y)$, we interpret $(p_x,p_y)$ as ``the restriction of $p_y$ to $\ran(p_x)$".  

Now the constellation $Q=C_E(T_X^0)$ correctly enumerates the elements of $PT_X$, and the straightforwardly calculated constellation product in $Q$ correctly calculates compositions of partial functions whenever the former are defined.  For example, $(p_x,p_y)(1,i)$ is just the constellation product $(p_x,p_y)\circ(1,i)=(p_x,p_yi)=(p_x,p_x)$, corresponding to calculating $\{(x,y)\}$ followed by $i$ to give $\{(x,x)\}=e_x$.  However, for general  compositions, we use multiplication in $Rest_0(E,T_X^0)$.   Reversing the order in the previous example gives $(1,i)(p_x,p_y)$, which requires calculation of $i\cdot p_x$.  Now $Eq(ip_x,i)=Eq(p_x,i)=\{0,p_y\}=Sp_y$, so $i\cdot p_x=p_y$, and so
\[
(1,i)(p_x,p_y)=(1\wedge (i\cdot p_x),(i\cdot p_x)ip_y) =(p_y,p_yip_y)=(p_y,p_xp_y)=(p_y,p_y),
\]
mirroring the computation of $p_x$ followed by $j$ in $PT_X$ (which equals $p_y$).

\subsection{Binary relations under demonic composition.} \label{demoneg}

$PT_X$ is the canonical example of a left restriction semigroup, but other naturally arising examples exist.  Next we consider binary relations under the operations of domain and so-called {\em demonic composition}, to be introduced shortly.

First, we extend the definitions of domain and range to cover binary relations in the usual way: for $s\in Rel_X$, let 
\begin{align*}
\dom(s)&=\{x\in X\mid (x,y)\in s\mbox{ for some }y\in X\},\\
\ran(s)&=\{y\in X\mid (x,y)\in s\mbox{ for some }x\in X\}.
\end{align*}
On $Rel_X$, define {\em demonic composition} $\circledast$ via 
$$\rho\circledast\tau=\{(x,y)\in \rho\tau\mid \mbox{ for all } z\in X, (x,z)\in \rho\Rightarrow z\in \dom(\tau)\},$$
where $\rho\tau$ denotes the usual composition of $\rho,\tau\in Rel_X$.
As noted by various authors, $\circledast$ is associative; indeed if we define $D$ on $Rel_X$ exactly as for $PT_X$ in terms of domains, $(Rel_X,\circledast,D)$ is a left restriction monoid (for example, see \cite{DAD}) with identity the diagonal relation (the identity function on $X$).

Denote by $TRel_X$ the set of {\em left total} binary relations on $X$: those $s\in Rel_X$ for which $\dom(s)=X$.  Usual and demonic composition coincide on $TRel_X$, which is therefore a submonoid of $Rel_X$ under each of these operations.  Note that $T_X$ is a submonoid of $TRel_X$.  All of these observations carry over when we adjoin $0$ to the various semigroups.

\begin{thm}  \label{demonicmod}
	$(Rel_X,\circledast,D)\cong Rest_0(E,TRel^0_X)$, where $E$ is maximal right pre-reduced within $E(T_X^0)$ as in Example \ref{pareg}.
\end{thm}
\begin{proof} 
Since $D(Rel_X)=D(PT_X)$, $Rel_X$ has enough large idempotents (the same as for $PT_X$), and obviously satisfies $D(0)=0$.  Evidently $(Rel_X)_1$ gives a copy of $TRel_X$, and so the result follows by Theorem \ref{char0} on letting $E$ be chosen as in Theorem \ref{TX0extend}.
\end{proof}  

Though not a left restriction semigroup, the demigroup $S=(Rel_X,\cdot,D)$, in which $\cdot$ is ordinary relational composition and $D$ is domain, is a left Ehresmann semigroup, meaning that it satisfies certain weaker conditions than being a left restriction semigroup. It was noted in \cite{demonize} that $Rest^D_{D(S)}(S)$, the left $(D,D(S))$-completion of $S$, is isomorphic to $(Rel_X,\circledast,D)$, giving a different way to obtain the latter, not as an extension of $Rel_X$ but in effect by redefining multiplication on $Rel_X$ itself.

Note also that multiplication in the constellation $C_E(TRel^0_X)$ agrees with the restricted product (defined for any demigroup at the end of Section \ref{clrs}) obtained from either $(Rel_X,\cdot,D)$ or $(Rel_X,\circledast,D)$ (since standard and demonic composition of $\rho,\tau$ agree when $\ran(\rho)\subseteq \dom(\tau))$.  However, the semigroup left $E$-completion is only capable of reproducing the demonic composition.

Because $E$ is the same as in Theorem \ref{TX0extend}, it may be chosen to be right reduced, and then we obtain the following.

\begin{cor}
Let $E$ be a right reduced maximal right pre-reduced subset of $E(T_X^0)$.  Then $(Rel_X,\circledast,D)$ embeds as the largest left restriction semigroup with zero within the Zappa-Sz\'{e}p product $E\bowtie TRel_X^0$ having domains elements ${\hat E}$ as in Proposition \ref{ZSP}.
\end{cor}

We note that the choice of $E$ in Theorem \ref{demonicmod} is never maximal right pre-reduced in $E((TRel_X)^0)$.  For example, for the full relation $\nabla$ on $X$, there is no $e\in E$ for which $\nabla\, {\mc L}\, e$, as is easily seen.   However, this does not imply $TRel_X^0$ is not protomodal.

\begin{eg}  Constructing $Rel_X$ from $TRel_X^0$ when $|X|=2$.   \label{2con2}
\end{eg}

Again, let $X=\{x,y\}$.  This time, $S=TRel_X^0=T_X^0\cup \{g,h,\nabla,a,b\}$, where $T_X^0$ is represented as in Example \ref{2con}, and 
$g=\{(x,x),(x,y),(y,y)\}$, $h=\{(x,x),(y,x),(y,y)\}$, $\nabla$ is the full relation, $a=\{(x,x),(x,y),(y,x)\}$, and $b=\{(x,y),(y,x),(y,y)\}$.  Now $E(S)=\{0,1,e,f,g,h,\nabla\}$ (which happens to be right pre-reduced but not right reduced since $eh=e$ yet $he=\nabla$), and let $E$ be $\{0,1,e,f\}$ as in Example \ref{2con}. There are 16 elements of $Rest_0(S,E)$, namely the nine elements of $Rest_0(T_X^0,E)$, 
\[
(0,0), (e,e), (f,f), (e,f), (f,e), (1,1), (1,i), (1,e), (1,f), 
\]
together with five additional elements capturing the non-functional left-total relations over $X$,
\[
(1,g),(1,h),(1,\nabla),(1,a),(1,b),
\]
and the two further elements $(e,\nabla), (f,\nabla)$, corresponding to restricting the domain of $\nabla$ to $\dom(e),\dom(f)$ respectively, giving $\{(x,x),(x,y)\}$ and $\{(y,x),(y,y)\}$.  

Again, whenever it exists the constellation product in $C_E(TRel_X^0)$ faithfully calculates not only demonic but also standard compositions of relations.  For general demonic compositions, we use multiplication in $Rest_0(E,TRel_X^0)$.  As an example, consider $a\circledast \{(x,x),(x,y)\}=\{(y,x),(y,y)\}$.   Working instead in $Rest_0(E,TRel_X^0)$ gives $(1,a)(e,\nabla)=(1\wedge a\cdot e,(a\cdot e)\nabla)$.  But $Eq(ae,a)=Eq(e,a)=\{0,f\}=Sf$, so $a\cdot e=f$ and we obtain $(1,a)(e,\nabla)=(f,f\nabla)=(f,\nabla)$, which does indeed correspond to 
$\{(y,x),(y,y)\}$.

Note that $TRel_X^0$ in this example is not protomodal: it is routine to verify that the equalizing set $Eq(a,a\nabla)=Eq(a,\nabla)=\{a,\nabla,e, h\}$ is not generated as a right ideal by any member of $E(S)$.  

Indeed one can check that for any other idempotent in $E(TRel_X^0)$ aside from those in $E$, there exists $s\in TRel_X^0$ for which $Eq(s,se)$ is not generated by any member of $E(S)$.  It follows that $E$ is already largest possible (as in Proposition \ref{maxmodE}): no bigger $E'$ exists for which $TRel_X^0$ is an inductive $E'$-monoid.  We do not know if $E$ as in Theorem \ref{demonicmod} is largest possible for arbitrary $X$.

\subsection{Transformations and left total partitions.} \label{partiteg}

In Subsection \ref{pareg}, $T_X^0$ was seen to be a (left) protomodal monoid, and $PT_X\cong Rest_0(T_X^0)$ as a left restriction monoid.  In fact, $T_X$ is also right protomodal, leading to an unexpected connection with the partition monoid on the set $X$.  

We do not here intend to give an exhaustive introduction to partition monoids, and instead point the interested reader to introductions to the topic appearing in the literature, for example \cite{East} and \cite{eastgray}, where the notation and definitions are consistent with those we use here.  Here, we simply note that the partition monoid concept came from work in \cite{halram} and \cite{wilcox}, where in the finite case the authors of these papers related partition monoids to partition algebras (already considered by many authors), via the notion of a twisted semigroup algebra.

Let $X$ be a non-empty set.  A {\em partition} is an equivalence relation on $X\cup X'$ where $X'$ is a copy of $X$ with $x\in X$ corresponding to $x'\in X'$. 

It is convenient to draw partitions as graphs with vertex set $X\cup X'$ which have as their connected components the blocks of the partition, with vertices in $X$ appearing in an upper row and vertices in $X'$ below them.  For example, in Figure \ref{fig1}, in which $X=\{1,2,3,4,5\}$ and $X'=\{1',2',3',4',5'\}$, the first graph represents the partition $\rho$ given by $\{1,2,1'\},\{3,4,5,4',5'\},\{2',3'\}$. The second represents the transformation in which $1\mapsto 1, 2\mapsto 1,3\mapsto 4,4\mapsto 4,5\mapsto 5$ as the partition $t\in T_X$ given by $$\{1,2,1'\}, \{3,4,4'\},\{5,5'\}, \{2'\},\{3'\}$$ (note that there is precisely one element of $X'$ in every cell), and the third represents the equivalence relation determined by the partition of $X$ having cells $\{1\},\{2,3\},\{4,5\}$ as the partition $e\in F$ given by $\{1,1'\}, \{2,3,2',3'\}, \{4,5,4',5'\}$.   

\begin{figure}[h]
\begin{center}
\begin{tikzpicture}[scale=.6]

\begin{scope}[shift={(0,0)}]	
\uvs{1,...,5}
\lvs{1,...,5}
\stline11
\stline21
\stline34
\stline55
\uarc34
\darc23
\darc45
\end{scope}

\begin{scope}[shift={(8,0)}]	
\uvs{1,...,5}
\lvs{1,...,5}
\stline11
\stline21
\stline34
\stline44
\stline55
\end{scope}

\begin{scope}[shift={(16,0)}]	
\uvs{1,...,5}
\lvs{1,...,5}
\stline11
\stline22
\stline44
\uarc23
\uarc45
\darc23
\darc45
\end{scope}

\end{tikzpicture}
\caption{Graphical depictions of the partitions $\rho$ (left), $t$ (middle) and $e$ (right).} 
\label{fig1}
\end{center}
\end{figure}

The partitions on $X$ may be endowed with a special type of multiplication which makes them into a monoid, called the {\em partition monoid} on $X$, here denoted $P_X$.  Multiplication may be defined as follows.  Given $s,t\in P_X$, let $s^{\downarrow}$ and $t^{\uparrow}$ be the graphs obtained by changing every lower vertex $v'$ of $s$ and every upper vertex $v$ of $t$ into $v''$, let $G_{s,t}$ be the graph on the vertex set $X\cup X''\cup X'$ with edge set the union of the edge sets of $s^{\downarrow}$ and $t^{\downarrow}$, viewing $X''$ as the middle row of $G_{s,t}$. To obtain the product $st$ if $s$ and $t$ are in $P_X$, simply put $u,v\in X\cup X'$ in the same block if and only if there is a path from $u$ to $v$ in $G_{s,t}$. As an example, consider the product $e\rho$, where $e$ and $\rho$ are as in the examples depicted in Figure \ref{fig1}; this is depicted in Figure \ref{fig2}.  The first figure indicates the process of forming the product using $G_{e,\rho}$, with the second figure giving $e\rho$; clearly, $\rho t$ is given by the partition
$$\{1,2,3,4,5,1',4',5'\}, \{2',3'\}.$$

Note that the identity element $1$ in $P_X$ is the equivalence relation in which each cell is a pair $\{x,x'\}$ where $x\in X$; in the five-element base set example above, this is represented by the partition $\{1,1'\},\{2,2'\},\ldots,\{5,5'\}$.  Again, we refer the reader to work such as \cite{East} and \cite{eastgray} for more details.

\begin{figure}[h]
\begin{center}
\begin{tikzpicture}[scale=.6]

\begin{scope}[shift={(0,0)}]	
\uvs{1,...,5}
\mvs{1,...,5}
\lvs{1,...,5}
\stlinea11
\stlinea22
\stlinea44
\stlineb11
\stlineb21
\stlineb34
\stlineb55
\uarc23
\uarc45
\marc23
\marc34
\marc45
\darc23
\darc45

\end{scope}

\begin{scope}[shift={(8,0)}]	
\uvs{1,...,5}
\lvs{1,...,5}
\stline11
\stline21
\stline55
\uarc23
\uarc34
\uarc45
\darc23
\darc45
\end{scope}

\end{tikzpicture}
\caption{Forming $e\rho$ with $e,\rho$ as in Figure \ref{fig1}: $G_{e,\rho}$ (left) and $e\rho$ (right).}
\label{fig2}
\end{center}
\end{figure}

Now denote by $P^{lt}_X$ the submonoid of $P_X$ consisting of the {\em left total partitions}, meaning partitions having domain all of $X$: $\rho$ is one such if for all $x\in X$ there is $y\in X'$ such that $(x,y)\in \rho$ (this is denoted $P^{fd}_X$ in \cite{eastgray}).  Let $F$ consist of the idempotents in the copy of the dual symmetric inverse monoid (see \cite{fitzleach}) sitting in $P^{lt}_X$; the elements of this dual symmetric inverse monoid copy are the partitions in which no cell (equivalence class) consists solely of elements of just $X$ or just $X'$, and the elements of $F$ correspond with the equivalence relations on $X$.  A typical cell has the form $Y\cup Y'$ where $Y\subseteq X$ and $Y'=\{y'\in Y'\mid y\in Y\}\subseteq X'$.  Note that $F$ is a semilattice with multiplication corresponding to join of equivalence relations, and so for $e,f\in E$, $e$ is finer than $f$ if and only if $e\ \omega\ f$, which means that $ef=f=fe$.  It is standard to view $T_X$ as embedded in $P^{lt}_X$ by identifying $t\in T_X$ with the partition $\rho_t$ which puts those elements of $X'$ that are not in the range of $t$ into a block by itself. All three partitions depicted in Figure \ref{fig1} are left total.  

The following is noted in \cite{eastgray} (see Proposition $4.14$ there). 

\begin{pro}
$P^{lt}_X$ is a submonoid without zero of $P_X$ which is a right restriction monoid in which $R(P^{lt}_X)=F$.
\end{pro}
\begin{proof}   
For all $\rho\in P^{lt}_X$, let $R(\rho)$ be the coarsest $e\in F$ such that $\rho e=\rho$, which is the member of $F$ corresponding to the restriction of $\rho$ to $X'$.  (For example, with $\rho$ as in Figure \ref{fig1}, this is $\{1,1'\},\{2,3,2',3'\},\{4,5,4',5'\}$.)  This is the largest $e\in F$ with respect to $\omega$ such that $\rho e=\rho$.  But $F$ is a meet-semilattice under $\omega$ in which partition multiplication (which in this case is just the usual join of the corresponding equivalence relations) is this meet.  As in the discussion of Section \ref{leftrestetc}, $R$ defines an RC-semigroup operation on $M=P^{lt}_X$ in the sense of \cite{Csemi}, with $R(M)=F$.  Note that for any $t\in T_X$ and $e\in F$, $R(te)=e$.   It therefore only remains to show that for all $\rho,\tau\in P^{lt}_X$, $R(\rho)\tau=\tau R(\rho\tau)$. To do this, it suffices to show that $R(\rho\tau)=R(R(\rho)\tau)$, and for all $e\in R(P^{lt}_X)$ and $\rho\in S$, $e\rho=\rho R(e\rho)$.  

For $\rho\in P^{lt}_X$, let $t$ be any one transformation that takes everything in each domain block of $\rho$ to one particular element of the corresponding range block; it then follows fairly easily that $\rho=tR(\rho)$.  In particular, for $s\in T_X$ and $e\in R(P^{lt}_X)$, $es\in P^{lt}_X$ can have $t$ chosen to be $s$ itself, and so $es=sR(es)$.  Also, for any $e\in R(P^{lt}_X)$ and $\rho=tR(\rho)$ for some $t\in T_X$, $R(\rho e)=R(tR(\rho)e)=R(\rho)e$.

Hence for $\rho,\tau\in P^{lt}_X$, writing $\rho=sR(\rho)$ and $\tau=tR(\tau)$ for suitable $s,t\in T_X$, we see that
$$R(R(\rho)\tau)=R(R(\rho)tR(\tau))=R(R(\rho)t)R(\tau),$$
while
\bea
R(\rho\tau)&=&R(sR(\rho)tR(\tau))\\
&=&R(stR(R(\rho)t)R(\tau))\\
&=&R(R(\rho)t)R(\tau)\\
&=&R(R(\rho)tR(\tau))\\
&=&R(R(\rho)\tau),
\eea
and for any $e\in R(P^{lt}_X)$,
$e\rho=esR(\rho)=sR(es)R(\rho)=sR(\rho)R(esR(\rho))=\rho R(e\rho)$.
\end{proof}

In Figure \ref{fig1}, observe that $R(\rho)=e$, while $R(t)=1$, and of course $R(e)=e$.  From Figure \ref{fig2}, note that $R(e\rho)=\{1,4,5,1',4',5'\}$, depicted in Figure \ref{fig3}, and then $\rho R(e\rho)$ is also as in Figure \ref{fig3}; evidently this agrees with $e\rho$ as in Figure \ref{fig2}, in accord with the right restriction semigroup law $R(s)t=tR(st)$.

\begin{figure}[h]
\begin{center}
\begin{tikzpicture}[scale=.6]

\begin{scope}[shift={(0,0)}]	
\uvs{1,...,5}
\lvs{1,...,5}
\stline11
\stline22
\stline41
\stline55
\uarc23
\uarc45
\darc23
\darc45
\end{scope}

\begin{scope}[shift={(8,0)}]	
\uvs{1,...,5}
\mvs{1,...,5}
\lvs{1,...,5}
\stlinea11
\stlinea21
\stlinea34
\stlinea55
\uarc23
\uarc45
\marc23
\marc45
\stlineb11
\stlineb22
\stlineb41
\stlineb55
\darc23
\darc45
\end{scope}

\begin{scope}[shift={(16,0)}]	
\uvs{1,...,5}
\lvs{1,...,5}
\stline11
\stline21
\stline55
\uarc23
\uarc34
\uarc45
\darc23
\darc45
\end{scope}

\end{tikzpicture}
\caption{Forming $\rho R(e\rho)$: $R(e\rho)$ (left), $G_{\rho,R(e\rho)}$ (centre) and $\rho R(e\rho)$ (right).}
\label{fig3}
\end{center}
\end{figure}

Generally, in the right restriction monoid $P^{lt}_X$, $(P^{lt}_X)_1$ consists of those $\rho\in P^{lt}_X$ such that $R(\rho)=1$, that is, those $\rho$ for which $\rho e=\rho$ for some $e\in E$ implies $e=1$.  It follows that such a $\rho$ must be a left total partition such that $\rho$ restricted to its range (lower) row must be the diagonal relation.  So the elements of $(P^{lt}_X)_1$ may be viewed as transformations (with those $x\in X$ not in the image of $\rho$ being in cells by themselves), as for $t$ in Figure \ref{fig1}.  Composition in $(P^{lt}_X)_1$ is then easily seen to correspond to transformation composition, so $(P^{lt}_X)_1\cong T_X$, and we identify the two monoids.   

Again, recall that the idempotents of $T_X$ are projections, but this time we are interested in $\omega^r$ on $E(T_X)$.  Recall that the kernel of $s\in T_X$ is the equivalence relation $\{(x,y)\in X\times X\mid xs=ys\}$.  It is well-known that for $e,f\in E(T_X)$, $e\ \omega^r\ f$ (that is, $fe=e$) if and only if the kernel of $e$ contains the kernel of $f$.  Hence $E\subseteq E(T_X)$ is maximal left pre-reduced if and only if it consists of precisely one projection for every possible element of $F\subseteq E(P^{lt}_X)$.

\begin{thm}  \label{ltpart}
$T_X$ is a right protomodal monoid, and $P^{lt}_X\cong RRest(T_X)$. 
\end{thm}
\begin{proof}  For $e\in F=R(P^{lt}_X)$, there is $e'\in (P^{lt}_X)_1=T_X$ such that $ee'=e',e'e=e$: let $e'$ be any projection having kernel the equivalence relation associated with $e$ (with elements of $X$ not in the range in cells on their own).  So $P^{lt}_X$ has enough large idempotents (noting it has no zero), and so by (the dual of) Theorem \ref{char0}, 
$P^{lt}_X\cong RRest((P^{lt}_X)_1,E')$, where $E'=\{e'\mid e\in R(P^{lt}_X)\}$ is left pre-reduced. But for each $f\in E(T_X)$, select $e\in F$ such that $e\, {\mc L}\, f$ (so $e$ is the kernel of $f$); this shows $P^{lt}_X$ has precisely enough large idempotents, so arguing as in Theorem \ref{TX0extend}, $E'$ is maximal left pre-reduced and so $T_X$ is right protomodal, and we may write $P^{lt}_X\cong RRest(T_X)$. 
\end{proof}

By (the dual of) the proof of Theorem \ref{char0}, there is an isomorphism $\theta: RRest(T_X,E)\rightarrow P^{lt}_X$, where $E$ consists of precisely one projection for each possible member of $F$, given by $(t,e')\theta=te$ for all $t\in T_X$ and $e'\in E$.  As an example, if $e'$ is the projection in $E$ corresponding to $e$ in Figure \ref{fig1}, then $te'=t$, and indeed $(t,e')\theta=te=\rho$.  It was noted in \cite{eastgray} that one could (non-uniquely) factorize elements of $P^{lt}_X$ into products of elements of $T_X=(P^{lt}_X)_1$ with elements of $F=R(P^{lt}_X)$, and indeed this is true for any right restriction semigroup with enough large idempotents; see (the dual of) Corollary \ref{factor}.

The opposite monoid $S$ of $P^{lt}_X$ is a left restriction monoid, hence is functionally representable.  But even more can be said, because $P^{lt}_X$ comes equipped with a domain operation given by $D(s)=ss^*\in R(P^{lt}_X)$; dualizing gives a range operation on the left restriction monoid $S$, which can easily be shown to satisfy the laws given in \cite{scheinDR} that characterise semigroups of partial functions equipped with $D$ and $R$.  So $P^{lt}_X$ can be viewed as a kind of dual structure to $PT_X$.  This extends the relationship between the dual symmetric inverse monoid and the familiar symmetric inverse monoid $I_X$ considered in \cite{fitzleach} (noting that these embed as right restriction monoids in $P^{lt}_X$ and $PT_X$ respectively), and was also noted in \cite{eastgray} since both arise within $P_X$.  It remains to be determined whether there is a formal categorical duality as for symmetric and dual symmetric inverse monoids as in \cite{fitzleach}.

As in the previous example, we can re-phrase things in terms of Zappa-Sz\'{e}p products since there are choices of maximal left pre-reduced $E$ in $E(T_X)$ that are left reduced.  Again, the well-ordering theorem is needed for this.  Let $X$ be a set, and equip it with a well-order $\leq$.  Define $E\subseteq T_X$ such that, for each partition of $X$, we form the projection which maps each cell of the partition to its smallest member.  Evidently there is precisely one member of $E$ for each equivalence relation on $X$, so $E$ is maximal left pre-reduced.  

\begin{lem}  \label{llE}
The set $E\subseteq E(T_X)$ defined above is left reduced.
\end{lem}
\begin{proof}
Suppose $e,f\in E$ and $ef=f$, so $ker(e)\subseteq ker(f)$.  Now for $x\in \ran(e)$, $e^{-1}(x)$ is a cell of the equivalence relation $ker(e)=\{(a,b)\in X\times X\mid e(a)=e(b)\}$, which is the disjoint union of cells of $ker(f)$, one of which therefore contains $x$.  Since $x$ was the smallest member of its cell in $ker(e)$, it will be the smallest member of the not larger cell in $ker(f)$ containing it, so by definition $x\in \ran(f)$.  Hence $\ran(e)\subseteq \ran(f)$, and so $fe=f$.  Hence $E$ is left reduced.
\end{proof}

\begin{cor}
Let $E$ be a left reduced maximal left pre-reduced subset of $E(T_X^0)$.  Then $P^{lt}_X$ embeds as the largest right restriction semigroup within the Zappa-Sz\'{e}p product $T_X\bowtie E$ (defined in the obvious dual way to $E\bowtie S$) with set of range elements ${\hat E}$ as in (the dual of) Proposition \ref{ZSP}.
\end{cor}
\begin{proof} 
$P^{lt}_X\cong RRest(E,T_X)\subseteq T_X\bowtie E$.  Now use the dual of Proposition \ref{max}.
\end{proof}

\section{Further work}  \label{more}

Operations on the monoid $PT_X$ other than domain have been considered, for example intersection, range and antidomain.  Each is entirely determined by the structure of the left restriction semigroup, since a given left restriction semigroup admits an operation modelling intersection, range or antidomain in at most one way.  Hence, it must be possible to determine properties of inductive left $E$-monoids that give rise to these operations in their left $E$-completions.  In future work, we shall show that an operation behaving like intersection exists in $Rest(E,S)$ (resp. $Rest_0(E,S)$ if $S$ is integral) if and only if $S$ is such that left equalizers of any two {\em arbitrary} elements exist, and moreover this assumption forces $S$ to be protomodal.

In the case of $Rel_X$ as in Example \ref{demoneg}, other operations and relations, such as union and inclusion, can be defined on $TRel_X$, and should extend to $Rel_X$ to give demonic counterparts of union and inclusion.  Again, we plan to explore this in future work.

\section*{Acknowledgements}

I would like to thank James East for discussions relating to the submonoid $P^{lt}_X$ of $P_X$, as well as for assistance with the generation of figures to represent partitions on finite sets.  I am also extremely appreciative of the anonymous referee's tireless efforts in furnishing so many useful comments and suggestions, which greatly helped to increase the readability of the work.

\vfill

\noindent Tim Stokes\\
Department of Mathematics\\
University of Waikato\\
Hamilton 3216\\
New Zealand.\\
email: tim.stokes@waikato.ac.nz

\end{document}